%% file: l-graphs.tex
\pdfoutput=1
\documentclass[12pt]{amsart}

\usepackage[margin=1.25in]{geometry}
\usepackage{amsthm, amsmath, amsfonts, cite, enumitem}
\usepackage{thm-restate}
\usepackage{caption}
\usepackage[utf8]{inputenc}
\usepackage{tikz,pgfplots}
\pgfplotsset{compat=1.15}
\usetikzlibrary{decorations.markings}
\usetikzlibrary{fillbetween}
\usetikzlibrary{external}
\usepackage{enumitem}
\usepackage{amssymb}
\usepackage[unicode]{hyperref}
\usepackage{bookmark}

\hypersetup{colorlinks=true,linkcolor=black,citecolor=black,filecolor=black,urlcolor=black,bookmarksdepth=section}

\setlist[itemize]{leftmargin=0.4in}

\usetikzlibrary{arrows}
\usetikzlibrary{patterns}
\usetikzlibrary{shapes,snakes}
\usetikzlibrary{arrows}
\usetikzlibrary{patterns}
\usetikzlibrary{shapes,snakes}
\usetikzlibrary{calc}

\usetikzlibrary{hobby}

\newtheorem{theorem}{Theorem}
\newtheorem{proposition}[theorem]{Proposition}
\newtheorem{conjecture}[theorem]{Conjecture}
\newtheorem{lemma}[theorem]{Lemma}
\newtheorem{claim}{Claim}[theorem]

\newcommand{\cal}[1]{\mathcal{#1}}

\newcommand{\drawL}[3]{\draw[thick] (#1,0) -- ++(0,#2) -- ++(#3,0);}
\newcommand{\drawLColoured}[4]{\draw[thick, #4] (#1,0) -- ++(0,#2) -- ++(#3,0);}
\newcommand{\drawLUltra}[3]{\draw[ultra thick] (#1,0) -- ++(0,#2) -- ++(#3,0);}

\let\leq\leqslant \let\le\leqslant
\let\geq\geqslant \let\ge\geqslant

\DeclareMathOperator{\dom}{dom}

\hypersetup{
  pdftitle={Grounded L-graphs are polynomially χ-bounded},
  pdfauthor={James Davies, Tomasz Krawczyk, Rose McCarty, Bartosz Walczak}
}

\title{Grounded L-graphs are polynomially $\chi$-bounded}
\author[Davies]{James Davies}
\author[Krawczyk]{Tomasz Krawczyk}
\author[McCarty]{Rose McCarty}
\author[Walczak]{Bartosz Walczak}
\date{}
\address[James Davies, Rose McCarty]{Department of Combinatorics and Optimization, School of Mathematics, University of Waterloo, Canada}
\address[Tomasz Krawczyk, Bartosz Walczak]{Department of Theoretical Computer Science, Faculty of Mathematics and Computer Science, Jagiellonian
University, Krak\'ow, Poland}

\thanks{Tomasz Krawczyk and Bartosz Walczak were partially supported by National Science Center of Poland grant 2015/17/D/ST1/00585.}

\begin{document}

\begin{abstract}
A \emph{grounded L-graph} is the intersection graph of a collection of ``L'' shapes whose topmost points belong to a common horizontal line. We prove that every grounded L-graph with clique number $\omega$ has chromatic number at most $17\omega^4$. This improves the doubly-exponential bound of McGuinness and generalizes the recent result that the class of circle graphs is polynomially $\chi$-bounded. We also survey $\chi$-boundedness problems for grounded geometric intersection graphs and give a high-level overview of recent techniques to obtain polynomial bounds.
\end{abstract}

\maketitle

\section{Introduction}
Following Gy{\'{a}}rf{\'{a}}s~\cite{Gyarfas87}, we say that a class of graphs $\mathcal{G}$ is \emph{$\chi$-bounded} if there is a function $f\colon \mathbb{N} \to \mathbb{N}$ such that the chromatic number $\chi$ of any graph in $\mathcal{G}$ is at most $f(\omega)$, where $\omega$ is the clique number of the graph. 
Additionally, $\mathcal{G}$ is \emph{polynomially $\chi$-bounded} if $f$ can be taken to be a polynomial function.
A flurry of research has been devoted to distinguishing which classes of graphs are $\chi$-bounded and which are polynomially $\chi$-bounded.
In particular, Esperet~\cite{esperet2017habilitation} conjectured that every $\chi$-bounded class of graphs that is closed under taking induced subgraphs is polynomially $\chi$-bounded.
We refer the reader to the recent survey on $\chi$-boundedness~\cite{ScottSeymour20} for more information.

An important branch of this research is focused on geometric intersection graphs.
The \emph{intersection graph} of a family of geometric objects has these objects as vertices and the pairs of objects with non-empty intersection as edges.
We restrict our attention to intersection graphs of objects in the plane, where by an \emph{object} we mean a compact arc-connected set in~$\mathbb{R}^2$.
By considering objects of some particular kind, for example, having a specific geometric shape, we obtain various classes of graphs. 

In this paper we focus on \emph{grounded L-graphs}, which are geometric intersection graphs of grounded L-shapes. 
An \emph{L-shape} consists of a vertical segment and a horizontal segment that meet at their uppermost and leftmost endpoints respectively, i.e., form an upside-down ``L''. 
An L-shape is \emph{grounded} if its lowermost point is on the $x$-axis. 
Grounded L-graphs were introduced by McGuinness~\cite{McGuinness96}, who showed that they satisfy $\chi = 2^{O(4^\omega)}$.
Our main result is that this class is in fact polynomially $\chi$-bounded.

\begin{theorem}
\label{thm:main}
Every grounded L-graph with clique number $\omega$ has chromatic number at most $17\omega^4$.
\end{theorem}

Theorem~\ref{thm:main} implies the recent breakthrough result of Davies and McCarty~\cite{DaviesMcCarty} that the class of circle graphs is polynomially $\chi$-bounded (albeit with a worse bound). 
Our proof appropriately adapts and extends those techniques.
It is tempting to see how far this method can be pushed; with this in mind, we will now survey $\chi$-boundedness problems for geometric intersection graphs. 
We refer the reader in advance to Figure~\ref{fig:graph_classes_hierarchy}, which depicts the graph classes that will be discussed as well as all possible inclusions between them and all of the best bounds on $\chi$ in terms of $\omega$.
We will give a high-level overview of the proof technique in Section~\ref{sec:prelim} in order to make this area more accessible.

\subsection*{Coloring geometric intersection graphs}

Due to practical and theoretical applications, the following graph classes have been particularly intensively studied:
\begin{itemize}[leftmargin=5.5mm]
\item \emph{rectangle graphs} -- intersection graphs of axis-parallel rectangles,
\item \emph{segment graphs} -- intersection graphs of segments,
\item \emph{L-graphs} -- intersection graphs of L-shapes,
\item \emph{string graphs} -- intersection graphs of strings, where a \emph{string} is a bounded continuous curve in the plane. 
\end{itemize}
The class of string graphs is the most general class of intersection graphs of objects in the plane, as every such graph 
can be easily represented as an intersection graph of strings ``filling'' these objects.

The study of chromatic properties of geometric intersection graphs was initiated by Asplund and Gr\"unbaum~\cite{AsplundGrunbaum60},
who proved that rectangle graphs satisfy $\chi \leq 4\omega^2-4\omega$ (when $\omega\geq 2$).
This was later improved to $\chi \leq 3\omega^2-2\omega-1$ by Hendler~\cite{Hen98}.
Chalermsook~\cite{Cha11} showed the bound $\chi = O(\omega\log{\omega})$ for the special case that none of the rectangles is contained in another.
Very recently, Chalermsook and Walczak~\cite{CW21} showed that rectangle graphs satisfy $\chi=O(\omega\log\omega)$, improving the quadratic bound that lasted for 60 years.
From below, Kostochka~\cite{Kos04} claimed a construction of rectangle intersection graphs satisfying $\chi = 3\omega$, 
and this remains the best known lower bound up to now.
Except for rectangle intersection graphs, $\chi$-boundedness is a rather 
rare feature among intersection graphs of geometric objects in the plane.
Pawlik et al.~\cite{PKK+13} proved that intersection graphs of objects obtained by independent horizontal and vertical scaling and translation of some 
fixed template object different than an axis-parallel rectangle, are not $\chi$-bounded. 
In particular, this shows that L-graphs, segment graphs, and string graphs are not $\chi$-bounded.
Fox and Pach~\cite{FP14} showed that every $n$-vertex string graph satisfies $\chi = (\log{n})^{O(\log{\omega})}$.
From below, Krawczyk and Walczak~\cite{KW17} constructed string graphs satisfying $\chi = \Theta_{\omega}((\log\log{n})^{\omega-1})$.
It is possible that all string graphs have chromatic number of order at most $(\log \log {n})^{f(\omega)}$ for some function $f\colon \mathbb{N} \to \mathbb{N}$.
So far, such a bound is proved for some particular, still not $\chi$-bounded, subclasses of string graphs~\cite{KPW15,KW17,Walczak20}. 

A quite different picture emerges when we consider intersection graphs of so-called ``grounded objects''.
Following terminology introduced by Cardinal et al.~\cite{CardinalFMTV18} (used in a similar form by Jel{\'{\i}}nek and T{\"{o}}pfer~\cite{JelTop19}), 
we say that a family of objects in the plane is \emph{grounded} if all these objects 
are contained in a closed half-plane and intersect the boundary of this half-plane, called a \emph{grounding line}\footnote{Note that for the class of grounded L-graphs, we may always take the grounding line to be the $x$-axis, and the half-plane to lie above the $x$-axis.}.
Jel{\'{\i}}nek and T{\"{o}}pfer~\cite{JelTop19} were the first who systematically studied relations (containments and separations) 
between most known classes of this kind. 
In particular, they considered the following classes of graphs:
\begin{itemize}[leftmargin=5.5mm]
 \item \emph{outerstring graphs} -- intersection graphs of \emph{grounded strings}, that is, strings whose one end intersects the horizontal grounding line
 and the remaining points are contained above the line.
 \item \emph{outer-$1$-string graphs} -- intersection graphs of simple families of grounded strings,
 where a family of strings is \emph{simple} if any two of them intersect in at most one point.
 \item \emph{grounded segment graphs} -- intersection graphs of \emph{grounded segments}, that is, segments so that one end intersects the grounding line.
 \item \emph{grounded $\{\mathrm{L},\text{\reflectbox{\upshape L}}\}$-graphs} -- intersection graphs of grounded L-shapes and their mirror images, grounded $\text{\reflectbox{L}}$-shapes, 
 where a \emph{grounded L-shape} is an L-shape with the bottommost point on the grounding line.
 An argument of Middendorf and Pfeiffer~\cite{MiddendorfP92} shows that grounded $\{\mathrm{L},\text{\reflectbox{L}}\}$-graphs form a subclass of grounded segment graphs.
 \item \emph{grounded L-graphs} -- intersection graphs of grounded L-shapes.
 \item \emph{circle graphs} -- intersection graphs of a family of chords in a circle.
 Circle graphs can be equivalently defined as intersection graphs of \emph{grounded semicircles}, that is, semicircles with both ends touching the grounding line.
 The class of circle graphs is contained in the class of grounded L-graphs -- see Figure~\ref{fig:circle_graphs}.
\input ./figures/circle_graphs.tex 
 \item \emph{permutation graphs} -- intersection graphs of a family of segments spanned between two horizontal lines.
 Permutation graphs can be equivalently defined as intersection graphs of semicircles which are pairwise grounded on intersecting intervals; as such, they
 are contained in the class of circle graphs -- see Figure~\ref{fig:permutation_interval_graphs}.
\input ./figures/permutation_interval_graphs.tex 
 \item \emph{monotone L-graphs}, also called as \emph{max point-tolerance graphs} -- 
 intersection graphs of a family of L-shapes all of whose bends belong to a common upward-sloping line.
 \item \emph{interval graphs} -- intersection graphs of intervals contained in the grounding line. 
 The class of interval graphs is contained in the class of grounded L-graphs and in the class of monotone L-graphs -- see Figure~\ref{fig:permutation_interval_graphs}.
\end{itemize}
In the context of the coloring problem, the following two classes of intersection graphs of grounded objects were considered:
\begin{itemize}[leftmargin=5.5mm]
 \item \emph{interval filament graphs} -- intersection graphs of \emph{interval filaments}, 
 which are continuous non-negative functions defined on closed intervals attaining zero values at
 their endpoints. 
 Interval filament graphs are contained in the class of outerstring graphs (two interval filaments may intersect in many points).
 \item \emph{polygon-circle graphs} -- intersection graphs of polygons inscribed into a circle.
 Polygon-circle graphs can be equivalently defined as intersection graphs of \emph{circle filaments}, which are 
 interval filaments consisting of some number of externally touching grounded semicircles.
 In particular, polygon-circle graphs form a subclass of interval filament graphs
 and extend the classes of circle graphs and interval graphs -- see Figure~\ref{fig:polygon_circle_graphs}.
\input ./figures/polygon_circle_graphs.tex 
\end{itemize}
Figure~\ref{fig:graph_classes_hierarchy} presents a diagram showing all possible inclusions between the graph classes defined above (so any inclusions which are not depicted are known to not exist). 
It also depicts the best known upper and lower bounds for $\chi$ in terms of $\omega$, to the best of our knowledge. In the rest of the introduction, we describe where the depicted inclusions/non-inclusions and upper/lower bounds come from. 
\input ./figures/graph_classes_ext.tex

It is commonly known that interval and permutation graphs are perfect\footnote{A graph is \emph{perfect} if $\chi=\omega$ holds for all of its induced subgraphs.} (see~\cite{golumbicBook}, for instance), 
which means that they satisfy $\chi = \omega$.
Circle graphs and monotone L-graphs are no longer perfect, and hence no other class shown in the diagram is perfect.

Catanazaro et al.~\cite{CatanzaroCFHHHS17} observed that monotone L-graphs can be alternatively defined as intersection graph of 
rectangles whose top-left vertices lie in a common upward-sloping line.
In such a representation no rectangle is contained in another, which proves that monotone L-graphs satisfy $\chi = O(\omega\log{\omega})$ by the result of Chalermsook~\cite{Cha11}.
Due to our knowledge, no non-trivial lower bound on the chromatic number in this class of graphs is known.

The study of chromatic properties of circle graphs was initiated by Gy{\'{a}}rf{\'{a}}s~\cite{Gyarfas85}, who proved $\chi = O(\omega^24^\omega)$ for this class of graphs.
This bound was later improved to $\chi = O(\omega ^22^\omega)$ by Kostochka~\cite{Kos88}, and to $\chi = O(2^\omega)$ by Kostochka and Kratochv\'{i}l~\cite{KK97}.
It is worth mentioning that the bound by Kostochka and Kratochv\'{i}l was actually obtained for polygon-circle graphs.
From below, Kostochka~\cite{Kos88,Kos04} constructed circle graphs that satisfy $\chi \ge \frac{1}{2}\omega(\ln \omega -2)$, and very recently Davies~\cite{davies2021circle} improved this slightly to~${\chi \ge \omega(\ln \omega -2)}$. 
This lower bound is currently the best in the class of outer-$1$-string graphs and polygon-circle graphs.
Grounded L-graphs were introduced by McGuinness~\cite{McGuinness96}, who showed that they satisfy $\chi  = 2^{O(4^\omega)}$~\cite{McGuinness96}.
McGuinness~\cite{McGuinness00} also proved that the chromatic number of triangle-free intersection graphs of simple\footnote{A family of grounded objects is \emph{simple} 
if the intersection of any subset of them is arc-connected.} families of grounded objects
is bounded by a constant~\cite{McGuinness00}.
Suk~\cite{Suk14} showed that intersection graphs of grounded simple $y$-monotone\footnote{A curve is \emph{$y$-monotone} if it intersects every horizontal line in at most one point.} curves satisfies $\chi=2^{O(5^\omega)}$.
In particular, Suk's result proves $\chi=2^{O(5^\omega)}$ for grounded segment graphs.
The recent two results were extended by Laso\'{n}, Micek, Pawlik, and Walczak~\cite{LMPW14}, who showed that intersection graphs 
of a family of simple grounded objects satisfy $\chi=2^{O(2^\omega)}$.
In particular, this proves $\chi=2^{O(2^\omega)}$ for outer-1-string graphs.
The research on the chromatic properties of intersection graphs of grounded objects has culminated in the work of Rok and Walczak~\cite{RW19}, 
who showed that outerstring graphs satisfy $\chi = 2^{O(2^{\omega(\omega-1)/ 2})}$.
Krawczyk and Walczak~\cite{KW17} considered colorings of interval filament graphs.
They showed that every interval filament graph satisfies $\chi \leq {\omega +1 \choose 2} \cdot \chi_{\mathrm{circ},\omega}$, 
where $\chi_{\mathrm{circ},\omega}$ is the maximum chromatic number of circle graph with clique number~$\omega$.
From below, Krawczyk and Walczak~\cite{KW17} constructed interval filament graphs that satisfy $\chi = \binom{\omega +1}{2}$, 
which seems to be the best lower bound known to date for outerstring graphs.

Note that the above-mentioned upper bounds on the chromatic number of graph classes extending circle graphs are at least exponential.
Introducing an entirely new technique, 
Davies and McCarty~\cite{DaviesMcCarty} showed that circle graphs satisfy $\chi = (1 + o (1)) \omega^2$,
thus obtaining the first polynomial upper bound on the chromatic number for a non-trivial class of intersection graphs of grounded objects.
Very recently Davies~\cite{davies2021circle} extended these techniques to further improve this bound to $(2+o (1)) \omega \log_2 \omega $. Thus circle graphs satisfy $\chi = \Theta(\omega \log \omega)$ by the aforementioned lower bound construction of Kostochka~\cite{Kos88,Kos04}.
Since circle graphs satisfy $\chi = O(\omega \log \omega)$, the above mentioned result of Krawczyk and Walczak proves 
that interval filament graphs (and hence polygon-circle graphs) satisfy $\chi = O(\omega^3 \log \omega)$.
Our main result adds to the list of known polynomially $\chi$-bounded classes of graphs. To be more precise,  Theorem~\ref{thm:main} shows that grounded L-graphs satisfy $\chi = O(\omega^4)$.

Clearly, the same asymptotic upper bound holds for grounded $\{\mathrm{L},\text{\reflectbox{L}}\}$-graphs 
(we color grounded L-shapes and grounded $\text{\reflectbox{L}}$-shapes with different sets of colors).
Furthermore, we also note that a simple twist in the proof of Davies and McCarty allows to  
show the bound $\chi = (1 + o (1)) \omega^2$ also for polygon-circle graphs.
Interestingly this appears to fail for the proof of the $(2+o (1)) \omega \log_2 \omega $ bound.
We will discuss this matter further, including the modifications required for proving the $(1 + o (1)) \omega^2$ bound for polygon-circle graphs in Section \ref{sec:prelim}.
Clearly, questions about polynomial $\chi$-boundedness for other graph classes shown in Figure~\ref{fig:graph_classes_hierarchy} remain open.
In particular, it is tempting to ask whether the existing techniques can be extended to prove the following conjecture.

\begin{conjecture}
The class of grounded segment graphs is polynomially $\chi$-bounded.
\end{conjecture}

Moreover, it would also be desirable to know a tight asymptotic bound on the chromatic number for classes of graphs that are polynomially $\chi$-bounded.
In this context, polygon-circle graphs, interval filament graphs, and monotone L-graphs are worth attention.

\subsection*{Containments and separations}

A quite important branch in the study of geometric intersection graphs concerns the questions related to containment and separation relations  
between such classes of graphs; see e.g.\ Cabello and Jej\v{c}i\v{c}~\cite {CJ17} as well as Cardinal, Felsner, Miltzow, Tompkins, and Vogtenhuber
\cite{CardinalFMTV18}.
As we already mentioned, such a study for intersection graphs of grounded objects was conducted by Jel{\'{\i}}nek and T{\"{o}}pfer~\cite{JelTop19}.
In particular, they indicated all inclusions for graph classes mentioned above, except for polygon-circle graphs and interval filament graphs.
We extend their work by placing in the right place these two graph classes.
To this end we prove the following separation results.
\begin{theorem}\
\label{thm:separation}
\begin{enumerate}
 \item There is an interval filament graph which is not a polygon-circle graph.
 \item There is a grounded L-graph which is not an interval filament graph.
 \item There is an L-monotone graph which is not an interval filament graph.
 \item There is a polygon-circle graph which is not an outer-$1$-string graph.
\end{enumerate}
\end{theorem}
Catanazaro et al.~\cite{CatanzaroCFHHHS17} showed a permutation graph which is not a monotone L-graph
(this shows that no superclass of permutation graphs is a subclass of monotone L-graphs). 
Cardinal et al.~\cite{CardinalFMTV18} constructed an outer-$1$-string graph which is not a grounded segment graph.
Jel{\'{\i}}nek and T{\"{o}}pfer~\cite{JelTop19} proved that:
\begin{itemize}
 \item there is a grounded segment graph which is not a grounded $\{\mathrm{L},\text{\reflectbox{L}}\}$-graph,
 \item there is a grounded $\{\mathrm{L},\text{\reflectbox{L}}\}$-graph which is not a grounded L-graph,
 \item there is a monotone L-graph which is not an outer-$1$-string graph.
\end{itemize}
All these results, together with Theorem \ref{thm:separation} and the folklore results concerning separations between interval, permutation, and circle graphs,
prove that there are no other inclusions between graph classes shown in Figure~\ref{fig:graph_classes_hierarchy} apart from those implied by the depicted arrows.

\section{Preliminaries}
\label{sec:prelim}

At a broad level our proof follows the strategy used by two of the authors to color circle graphs with clique number $\omega$ using $(1+o(1))\omega^2$ colors~\cite{DaviesMcCarty}. To explain the connection, and because it will be needed in our proof, we shall introduce circle and permutation graphs again from different perspectives.
We then sketch the proof strategy used for circle graphs in order to give an idea of how our proof for grounded L-graphs will go.
Afterwards we take a slight detour to sketch the modification required for extending the $(1+o(1))\omega^2$ bound to polygon-circle graphs and discuss the difficulty in extending the improved $O(\omega \log \omega)$ bound to polygon-circle graphs.
Then we discuss the strategy for grounded L-graphs, in particular we highlight some differences in the two proofs.
Lastly we introduce a few more definition and prove a simple reduction lemma for coloring grounded L-graphs.

\subsection*{Circle graphs}

Two open intervals \textit{overlap} if they intersect and neither is contained in the other. 
The \textit{overlap graph} of a set of open intervals ${\cal J}$ has vertex set ${\cal J}$, and two intervals are adjacent if they overlap. 
Notice that the class of circle graphs is exactly the class of overlap graphs of intervals.

Another equivalent definition of permutation graphs is that a graph is a permutation graph if there exist total orders $\prec_1$ and $\prec_2$ of 
its vertex set such that two distinct vertices are adjacent if and only if their relative order is different in $\prec_1$ and $\prec_2$. 
Permutation graphs are precisely the intersection graphs of collections of grounded L-shapes such that there exists a vertical line which intersects each L-shape exactly once. 
We will utilize this connection, giving proofs as needed.

The strategy for coloring overlap graphs of intervals $\mathcal{J}$ with clique number $\omega$ is to find some suitable finite collection of points $B\subseteq \mathbb{R}$ and to equip them with a total ordering $\prec$ and a coloring $c$. 
From such a triple $(B,\prec , c)$ we can then obtain a proper partial coloring of the overlap graph of $\mathcal{J}$ by first assigning each vertex $v$ to the first point in the total ordering $(B,\prec)$ that is contained in the interval corresponding to $v$. 
Then we give $v$ the same color as the point contained in $B$ that it is assigned to.
If we ensure that no pair of adjacent vertices are assigned to distinct points of the same color, then each monochromatic component consists of vertices assigned to a single point, and so induces a permutation graph. 
In this case we called the triple $(B,\prec , c)$ a pillar assignment (we shall use a different definition of pillar assignment for grounded L-graphs).

With this idea of pillar assignments in mind, at a basic level the proof strategy is to prove that a pillar assignment (which provides an improper partial coloring of our graph) 
that uses few colors can be extended to one that colors every vertex, still without using too many colors. 
The vertices assigned to each point of $B$ form well structured cuts of the graph, in the sense that the intervals corresponding to each connected component of the overlap graph induced by the uncolored vertices must be contained within some component of $\mathbb{R} \backslash B$, which we call a segment. 
This motivates the idea of the degree of a segment, which roughly speaking is the number of points in $B$ 
having assigned at least one interval from $\mathcal{J}$ with one endpoint in the segment. 
So the strategy is to prove that pillar assignments using a bounded number of colors and with bounded degree segments (with this bound being smaller) 
can be extended to a pillar assignment that colors every vertex, still without using too many colors.

Very roughly this is done as follows. 
Given a segment $S=(b^{-},b^{+})$ of bounded degree that we wish to extend the pillar assignment within, 
we begin by partitioning $S$ into a collection of smaller segments $\mathcal{S}= \{(b^{-},b_1),(b_1,b_2),\dots ,(b_{n-1},b^{+})\}$, each with a much smaller (but still suitably large) degree. 
From here we aim to use a divide-and-conquer argument to extend the pillar assignment $(B,\prec , c)$ to a new one  $(B^*,\prec^* , c^*)$ 
so that $\{b_1,\dots ,b_{n-1}\} \subset B^*$ and more vertices of the overlap graph are colored. 
The divide-and-conquer argument shall use colors distinct from those given to vertices with corresponding intervals having an endpoint in the segment $S$ 
(the number of which is measured by the degree of $S$). 
The number of colors required to carry out the divide-and-conquer argument is dependent on $|\mathcal{S}|$, 
and so if $|\mathcal{S}|$ is small enough then we succeed in extending the pillar assignment $(B,\prec , c)$ slightly.

The trickier case is when $|\mathcal{S}|$ is too large to carry out the divide-and-conquer argument. 
In this case we aim to derive a contradiction by finding a clique of size $\omega +1$.
This is done as follows.
First, we create a collection of intervals $\mathcal{I}$: for every point $b \in B$ and every interval $S'$ in $\mathcal{S}$, 
if $b$ contributes to the degree of $S'$, we add an interval $I$ with the endpoints in $b$ and in the middle of $S'$ to the set $\mathcal{I}$.
Moreover, for any such interval $I$ in $\mathcal{I}$ we choose an interval from $\mathcal{J}$ witnessing this contribution;
that is, an interval from $\mathcal{J}$ assigned to $b$ with an endpoint in $S'$.
Now, we complete the proof as follows. 
Since the degree of each segment contained in $\mathcal{S}$ is suitably large, and since $\mathcal{S}$ is large,
we can find in the set $\mathcal{I}$ a set of $\omega + 1$ pairwise overlapping intervals (see Lemma \ref{permutation bound}).
Eventually, the extremal lemma asserts that the witnesses associated to those intervals form a set of $\omega +1$ pairwise overlapping intervals in $\mathcal{J}$.
Such a set corresponds to a clique of size $\omega +1$ in the overlap graph of $\mathcal{J}$.

This completes the proof for circle graphs since we can always extend our pillar assignment slightly.

\subsection*{Polygon-circle graphs}

Before returning to grounded L-graphs, we briefly discuss extending bounds for circle graphs to polygon-circle graphs.
We begin with a sketch of how to tweak the proof of Davies and McCarty~\cite{DaviesMcCarty} that circle graphs have chromatic number at most $(1 + o(1))\omega^2$ 
to further apply to polygon-circle graphs.

By viewing polygon-circle graphs as intersection graphs of circle filaments, the pillars can be defined in essentially the same way as for circle graphs; 
filaments are assigned to the first point $b$ of the ordered pillars that is contained in the domain of the filament.
Then the induced subgraph corresponding to a pillar is again a perfect graph. 
With this notion of pillar, the divide-and-conquer argument is exactly the same as for circle graphs. 
The corresponding notion of degree of a segment for polygon-circle graphs is similar: it counts the number of pillars 
with at least one circle filament intersecting the segment assigned.

The main difference is then the extremal lemma required to show that the set $\mathcal{S}$ is not too large.
We create a set of intervals $\mathcal{I}$ analogously; that is, for every 
$b \in B$ and every $S' \in \mathcal{S}$ we add the interval with the endpoints in $b$ and in the middle of the interval $S'$ to $\mathcal{I}$ 
if $b$ contributes to the degree of $S'$. 
Clearly, this contribution is witnessed by a circle filament assigned to $b$ intersecting the segment $S'$. 
Here it is worth noting the main difference that distinguishes the proofs for polygon circle graphs and circle graphs:
a single circle filament might be a witness for many intervals from $\mathcal{I}$, but if this is the case, 
all of them share the same endpoint $b$ to which this circle filament is assigned.    
Again, if $\mathcal{S}$ and the degree of every segment in~$\mathcal{S}$ are large enough,
we can find a set of $\omega+1$ pairwise overlapping intervals in~$\mathcal{I}$.
The previous observation asserts that circle filaments witnessing those intervals are pairwise distinct.
Moreover, it is proved the same way that their domains are pairwise overlapping.
In particular, these circle filaments must be pairwise intersecting.

An old problem of Kostochka and Kratochvíl~\cite{KK97} asks whether the optimal $\chi$-bounding function for polygon-circle graphs is within a constant fraction of the optimal $\chi$-bounding function for circle graphs. Since circle graphs are now known to have an optimal $\chi$-bounding function of $\Theta (\omega \log \omega)$~\cite{davies2021circle}, two reasonable approaches to this problem would be to either improve the lower bound construction, or to extend the $O(\omega \log \omega)$ proof for circle graphs to polygon-circle graphs. However the latter does not appear to extend as easily as the proof of the $(1 + o(1))\omega^2$ bound.

The $O(\omega \log \omega)$ proof for circle graphs uses a notion of pillar assignment that directly obtains a proper coloring rather than an improper coloring which is later refined. The induced permutation graphs assigned to each pillar are colored as we go in a particularly well structured way so that certain configurations of intervals with corresponding vertices colored by the pillar assignment yield pairwise overlapping intervals elsewhere. This part is critical in allowing for the improved bound. However for polygon-circle graphs, the vertices assigned to each pillar would instead induce trapezoid graphs, a more complex class that contains permutation graphs. Although trapezoid graphs are again perfect, they appear to be more difficult to color in a well structured well that facilities the rest of the proof.

\subsection*{Grounded L-graphs}

Next we sketch some of the main differences in the strategy that we use for grounded L-graphs. The reader may wish to refer to Figures \ref{fig:pillar} and \ref{fig:pillarPhi} for an idea of the algorithm we use to draw pillars as well as the definition of pillar assignments for grounded L-graphs.

Now let $\cal L$ be a collection of grounded L-shapes whose intersection graph, which we denote by $G({\cal L})$, has clique number $\omega$. The basic strategy is the same, we shall use triple $(B,\prec, c)$ to color our grounded L-graph, where this time $B$ consists of points of the $x$-axis, which we will call ``pillar bases''. As before, we shall color vertices with the same color of the pillar base that their corresponding L-shape is assigned to.

What is not so obvious however is how one should decide which pillar base of $B$ a vertex corresponding to a given L-shape should be assigned to. Given the first pillar base $b$ in the totally ordered set $(B,\prec)$, we use an algorithm to define a curve $P_b$, called a ``pillar'', beginning at $b$ and going up and to the left. This algorithm is dependent on the collection of grounded L-shapes $\mathcal{L}$, and this curve ends up being a staircase which ends at a one-way infinite vertical segment (see Figure \ref{fig:pillar}). 
We then continue with the next pillar base $b'$ in the total order, defining its pillar $P_{b'}$ again with an algorithm that depends on $\mathcal{L}$.
Again the pillar $P_{b'}$ begins at $b'$ and is a staircase that goes up and to the left; we only stop $P_{b'}$ if it runs into an earlier pillar. This process is carried out for every pillar base of $B$ in order (see Figure \ref{fig:pillarPhi}).

The total order $\prec $ of $B$ also naturally orders the pillars $\mathcal{P} = \{P_b: b\in B\}$, and we can also consider $c$ as a coloring of $\mathcal{P}$. Then we assign each vertex of $G({\cal L})$ to the earliest pillar $P$ in the totally ordered set $(\mathcal{P}, \prec)$ that intersects its corresponding L-shape. Then for grounded L-graphs, we call such a triple $(\mathcal{P}, \prec , c)$ a pillar assignment.

This definition of pillars and pillar assignment may seem somewhat convoluted compared to that of circle graphs, but the algorithm defining the pillars is carefully designed so that certain configurations of L-shapes with corresponding vertices of certain colors given by the pillar assignment yield pairwise intersecting L-shapes elsewhere in $\mathcal{L}$. Crucially, this gives us another way to find cliques in the intersection graph rather than just finding them directly. This extra way of finding cliques is needed for one of our required extremal results.

For grounded L-graphs there are more ways that a grounded L-shape $L$ with corresponding vertex $v$ colored by our pillar assignment could intersect an L-shape $L'$ whose corresponding vertex $v'$ is uncolored. So instead of one notion for the degree of a segment we require a few that depend on the position of the L-shape contributing to the degree. This mainly consists of definitions for the ``left degree'' and the ``right degree'' of a segment. Then with extremal results that are tailored to our notion of pillar assignment and these individual notions of degree, the inductive argument on extending pillar assignments can be carried out in a way similar to that of circle graphs with another divide-and-conquer argument.

This pillar assignment approach could potentially be used to prove polynomial $\chi$-bounds for other classes of graphs, but it seems difficult to define a workable notion of a ``pillar''. We need to be able to color the induced subgraph corresponding to a ``pillar'', we need some kind of divide-and-conquer argument, and we need a suitable extremal lemma. All three of these steps are entirely dependent on the chosen notion of ``pillar''. So defining pillars seems like a major roadblock in using this approach to prove polynomial $\chi$-bounds.
Nevertheless we are hopeful that the approach can be further developed and applied to other classes of graphs,  particularly (but perhaps not necessarily exclusively) to other geometrically defined classes of graphs.

\subsection*{Flat L-collections}
We finish this section with some further definitions on grounded L-shapes, as well as a reduction that allows us to instead color simpler collections of L-shapes which we call ``flat L-collections''.

Let $L$ be an L-shape. We let $h(L)$ denote the $y$-coordinate of the topmost endpoint of the vertical segment of $L$, and we let $\ell(L)$ and $r(L)$ denote the $x$-coordinate of the leftmost and rightmost endpoints of the horizontal segment of $L$, respectively. So $L$ is totally determined by its \emph{height} $h(L)$, \emph{left endpoint} $\ell(L)$, and \emph{right endpoint} $r(L)$. We call the point $(\ell(L),h(L))$, where the vertical and horizontal segments meet, the \textit{corner} of $L$. Finally, we let $p(L)$ denote the closed interval $[\ell(L),r(L)]$; we call $p(L)$ the \emph{projection} of $L$ as it is the projection onto the $x$-axis.

We will consider collections $\cal L$ of grounded L-shapes such that for each distinct pair $L_1,L_2\in {\cal L}$, the intervals $p(L_1)$ and $p(L_2)$ do not share an endpoint and $h(L_1)\not= h(L_2)$. We call such a collection an \emph{L-collection} and write $G({\cal L})$ for its intersection graph. Notice that any grounded L-graph is isomorphic to some such $G({\cal L})$. We say that $\cal L$ is \emph{flat} if there are no $L_1, L_2 \in {\cal L}$ such that $p(L_1)\subseteq p(L_2)$ and $h(L_1) > h(L_2)$.  A \emph{flat grounded L-graph} is an intersection graph of a flat L-collection. We now give a simple reduction to flat L-collections.

\begin{figure}
    \centering
    \begin{tikzpicture}
        \def \h {.75} \def \w {6}
        \drawL{0}{\h}{\w}
        \drawL{1}{\h*2}{\w-2}
        \drawL{2}{\h*3}{\w-4}
        \draw (-0.5,0)--(7,0);
        \node[label=right:$L_3$] at (\w,\h) {};
        \node[label=right:$L_2$] at (\w-1,\h*2) {};
        \node[label=right:$L_1$] at (\w-2,\h*3) {};
    \end{tikzpicture}
    \caption{A chain of three L-shapes with $L_1 \preceq L_2 \preceq L_3$.}
    \label{fig:chain}
\end{figure}

\begin{lemma}\label{flattened}
    If $\cal L$ is an L-collection so that $G({\cal L})$ has clique number $\omega$, then there is an improper coloring of $\cal L$ with $\omega$ colors so that each color class is flat.
\end{lemma}
\begin{proof}
	We construct a partial order $({\cal L, \preceq})$. For two L-shapes $L_1$ and $L_2$, we have $L_1\preceq L_2$ if and only if $p(L_1)\subseteq p(L_2)$ and $h(L_1)> h(L_2)$, as in Figure~\ref{fig:chain}. It is straightforward to verify that $({\cal L, \preceq})$ is a partial order. Now notice that a chain in $({\cal L, \preceq})$ yields a clique in $G({\cal L})$; hence the maximum size of a chain is at most $\omega$. By the dual Dilworth Theorem, there is partition of $\cal L$ into at most $\omega$ antichains, and this gives the desired improper coloring.
\end{proof}

We remark that there are flat L-collections whose intersection graphs are not interval filament graphs. 
The example of a grounded L-graph that is not an interval filament graph given in the proof Theorem~\ref{thm:separation} is actually a flat grounded L-graph (see Figure~\ref{fig:flat_grounded_L_graph_not_inteval_filament}).
We mention this since, due to other work by the authors~\cite{DaviesMcCarty,KW17}, interval filament graphs are polynomially $\chi$-bounded.

\section{Pillar assignments}
\label{sec:pillarA}
In this section we will define a ``pillar assignment'', which encodes the points placed on the $x$-axis, the order in which they are placed, and their ``pillars''. We will also take care of condition 1) as outlined in the last section.

So let $\cal L$ be a flat L-collection. A \emph{base} $b$ of $\cal L$ is a point in $\mathbb{R}$ which is not an endpoint of $p(L)$ for any $L \in {\cal L}$. This is mainly to avoid any ambiguity. Given a set of bases $B$ and a total ordering $\prec$ of $B$ we will generate a collection of (simple) curves ${\cal P}=(P_b : b\in B)$. Each curve $P_b$ will be a ``staircase'' which starts at the point $(b,0)$ on the $x$-axis and goes up and to the left; refer to Figures~\ref{fig:pillar} and~\ref{fig:pillarPhi} for the following definitions.

Inductively suppose that $b\in B$ is the smallest base according to $\prec$ for which $P_b$ is not yet defined. Now starting at the point $(b,0)$ on the $x$-axis we draw the curve $P_b$ as follows.
\begin{itemize}\itemsep0em
	\item[1)] If the end of the curve we are drawing is contained in some previous curve $P_{b'}$ with $b'\prec b$, then we immediately stop drawing, thus completing $P_b$.
	
	\item[2)] If the end of the curve we are drawing is contained in the horizontal segment of an L-shape $L\in {\cal L}$ such that $b\in p(L)$, and the end of the curve is not the corner of $L$, then we continue going horizontally to the left until either 1) occurs or we reach the corner of $L$.
	
	\item[3)] Otherwise, we continue drawing the curve vertically upwards until either 1) or 2) occurs (or, if neither occurs again, then we complete $P_b$ with a one-way infinite segment).
\end{itemize}

\input ./figures/pillar.tex

We call the curves in $\cal P$ constructed in this way \emph{pillars}. For each $b\in B$, we call $b$ the \emph{base of $P_b$}. We give $\cal P$ the same total ordering as $B$; thus we use $({\cal P}, \prec)$ and $(B, \prec)$ interchangeably. We call such a  tuple $({\cal P}, \prec)$ \emph{ordered pillars} (and $(B, \prec)$ \emph{ordered bases}). We usually will not name the ordered bases; instead we will just talk about the ordered pillars $({\cal P}, \prec)$ and their bases.

Notice that a pillar $P$ is formed from horizontal and vertical segments (from steps 2) and 3) respectively). We say that an L-shape $L \in {\cal L}$ is a \emph{support} of a pillar $P$ if $L$ intersects $P$ in more than one point (equivalently, if $L$ caused $P$ to have a horizontal segment in some step 2)).

Next we have a key observation.

\begin{lemma}\label{pillar 2 permutations}
	Let $\cal L$ be a flat L-collection with ordered pillars $({\cal P}, \prec)$, and let $\omega$ be the clique number of $G({\cal L})$. Then for any pillar $P$, the subgraph of $G({\cal L})$ induced on the set of L-shapes which intersect $P$ is $2\omega$-colorable.
\end{lemma}
\begin{proof}
    Let ${\cal L}_1$ be the set of L-shapes that intersect a horizontal segment of $P$, and let ${\cal L}_2$ be the set of all other L-shapes that intersect $P$. It is enough to verify that for $i\in \{1,2\}$, $G({\cal L}_i)$ is a permutation graph.
    
    First we claim that, where $b$ is the base of $P$, each $L \in {\cal L}_1$ satisfies $b \in p(L)$. This is true if $L$ is a support of $P$, so we may assume $L$ is not a support. Then, as $L$ intersects a horizontal segment of $P$, there is a support $L'$ of $P$ so that $\ell(L') < \ell(L)$ and $h(L') < h(L)$. As $\cal L$ is flat, $r(L)>r(L')>b$, and so $b\in p(L)$ as required.
    
    Now we give ${\cal L}_1$ two total orderings $\prec^1$ and $\prec^2$. Let $\prec^1$ be the total ordering such that $L \prec^1 L'$ whenever $\ell(L) < \ell(L')$, and let $\prec^2$ be the total ordering such that $L \prec^2 L'$ whenever $h(L)> h(L')$. Thus $\prec^1$ orders ${\cal L}_1$ from left to right by where the L-shapes intersect the $x$-axis and $\prec^2$ orders ${\cal L}_1$ from top to bottom by where the L-shapes intersect the vertical line with $x$-coordinate $b$. It follows that two L-shapes in ${\cal L}_1$ intersect if and only if their relative orders are different in $\prec^1$ and $\prec^2$. So $G({\cal L}_1)$ is a permutation graph.
	
	The same idea works to show that $G({\cal L}_2)$ is a permutation graph; note that the horizontal segments of the L-shapes $L \in {\cal L}_2$ can be extended to the right until $b \in p(L)$ without changing the graph $G({\cal L}_2)$.
\end{proof}

We aim to use ordered pillars to color our intersection graph. Let $\cal L$ be a flat L-collection with ordered pillars $({\cal P}, \prec)$. We say that an L-shape $L\in {\cal L}$ is \emph{assigned} to a pillar $P\in {\cal P}$ if $L$ intersects $P$ and all other pillars in $\cal P$ that intersect $L$ occur later than $P$ in the total ordering $({\cal P}, \prec)$. Notice that each L-shape is assigned to at most one pillar of $\cal P$. Let $c: {\cal P} \to \mathbb{Z}^+$ be a coloring of $\cal P$ and let $\phi_{({\cal P}, \prec, c)}$ be the partial, improper coloring of $\cal L$ where an L-shape that is assigned to a pillar $P$ is given the color $c(P)$.

A \emph{pillar assignment} of a flat L-collection $\cal L$ is a tuple $({\cal P}, \prec, c)$ so that $({\cal P}, \prec)$ are ordered pillars and $c : {\cal P} \to \mathbb{Z}^+$ is a coloring such that if $L_1$ and $L_2$ are intersecting L-shapes assigned to distinct pillars $P_1$ and $P_2$ respectively, then $c(P_1)\not= c(P_2)$. A pillar assignment $({\cal P}, \prec, c)$ is \emph{complete} if every L-shape in $\cal L$ intersects (and so is assigned to) some pillar. So Figure~\ref{fig:pillarPhi} depicts a pillar assignment and the corresponding improper coloring of the L-shapes.

We have the following key lemma.

\input ./figures/pillar_assignment_and_pillar_coloring.tex

\begin{lemma}\label{colouring L with pillars}
	Let $\cal L$ be a flat L-collection such that $G({\cal L})$ has clique number at most $\omega$, and let $({\cal P}, \prec, c)$ be a complete pillar assignment of $\cal L$ using at most $k$ colors. Then the graph $G({\cal L})$ is $(2\omega k)$-colorable.
\end{lemma}

\begin{proof}
	Let $i$ be a color used by $\phi$ and let ${\cal L}_i=\{L\in {\cal L} : \phi_{({\cal P}, \prec, c)}(L)=i\}$. It will be enough to show that $G({\cal L}_i)$ is $2\omega$-colorable.
	
	By the definition of a pillar assignment, if two L-shapes in ${\cal L}$ intersect and are assigned to different pillars, then those pillars are colored differently. Hence the L-shapes of each connected component of $G({\cal L}_i)$ all intersect a single pillar. So by Lemma~\ref{pillar 2 permutations}, each connected component of ${G({\cal L}_i)}$ is $2\omega$-colorable and thus ${G({\cal L}_i)}$ itself is $2\omega$-colorable.
\end{proof}

At this point we have reduced the coloring problem to the problem of finding a complete pillar assignment of a flat L-collection. To be more precise, by Lemmas~\ref{flattened} and~\ref{colouring L with pillars}, if each flat L-collection whose intersection graph has clique number $\omega$ has a complete pillar assignment using at most $k$ colors, then each grounded L-graph with clique number at most $\omega$ is $(2\omega^2 k)$-colorable. 

Now we prove another important feature. The next lemma implies that if an L-shape $L$ is uncolored, then we can color it by placing a base anywhere in $p(L)$. Note that we write $p(L) \times [0,h(L)]$ for the square ``underneath'' $L$.

\begin{lemma}
    \label{lem:assigned}
    If $\cal L$ is a flat L-collection with ordered pillars $({\cal P}, \prec)$ and $L \in {\cal L}$ is such that $p(L) \times [0,h(L)]$ contains a point of a pillar $P$, then $L$ is assigned to a pillar which is at most $P$ according to $\prec$.
\end{lemma}
\begin{proof}
    It is enough to observe that if $P$ is the smallest pillar according to $\prec$ so that $p(L) \times [0,h(L)]$ contains a point in $P$, then $L$ must be assigned to $P$.
\end{proof}

We will apply the following divide-and-conquer lemma to a set of bases so that we may order and color them.
\begin{lemma}
    \label{divide-and-conquer}
    For any positive integer $k$ and set $B \subset \mathbb{R}$ of size at most $2^k-1$, there exist a total ordering $\prec$ and a $k$-coloring $c$ of $B$ so that for any $b_1, b_2 \in B$ with $c(b_1) = c(b_2)$ and $b_1<b_2$, there exists $b \in (b_1, b_2)\cap B$ so that $b\prec b_1$ and $b\prec b_2$.
\end{lemma}
\begin{proof}
    It is trivial to find such $\prec$ and $c$ for $k=1$. We proceed by induction on $k$. Choose $b\in B$ such that $|(-\infty,b)\cap B|,|(b,\infty )\cap B|\le 2^{k-1}-1$ and let $B_-= (-\infty,b)\cap B$ and $B_+=(b,\infty )\cap B$. Let $\prec_-$, $\prec_+$ and $c_-, c_+$ be such total orderings and $(k-1)$-colorings for $B_-$ and $B_+$ respectively (using the same set of colors, say $\{1,\dots , k-1 \})$. Now let $\prec$ be the total ordering obtained from $\prec_-\cup \prec_+$ by having $b$ precede all other elements (and each element in $B^-$ precede each element in $B^+$). Let $c$ be the $k$-coloring obtained from $c_-\cup c_+$ by setting $c(b)=k$. Now $\prec$ and $c$ provide the desired total ordering and coloring.
\end{proof}

\section{Extremal lemmas and degree}
\label{sec:extremal}
Let $\cal L$ be a flat L-collection with ordered pillars $({\cal P},\prec )$ with bases $B\subset \mathbb{R}$. A \emph{segment} of $B$ is an open interval with ends in $B\cup \{-\infty , \infty\}$ that contains no point in $B$. So there is a unique partition of $\mathbb{R}\backslash B$ into $|{\cal P}|+1$ segments.

Recall from Section~\ref{sec:prelim} that we wish to prove an extremal lemma regarding the ``degrees'' of disjoint intervals which are contained within a single segment. 
The ``degree'' of a segment $S$ is the number of pillars to which there is an L-shape $L$ assigned which satisfies either condition a) or condition b),
where:
\begin{itemize}\itemsep0em
    \item[a)] $L$ intersects the $x$-axis to the left of $b^-$ and intersects an uncolored L-shape grounded in $S$,
    \item[b)] $L$ intersects the $x$-axis at a point between $b^-$ and $b^+$.
\end{itemize}
So it is helpful to break up the ``degree'' into two parts: the ``left-degree'' corresponding to condition a) and the ``right-degree'' corresponding to condition b). 
Unfortunately, we also need a third type of degree, called the ``additional-degree'', for technical reasons. 
The issue is that when we break $S$ up into parts, the left-degree of a part can jump up suddenly. 
So we use the additional-degree to provide a buffer for the left-degree. 
For the following definitions we refer the reader to Figure~\ref{fig:pillardeg} for an example that illustrates the notion of the left and right degree.

For a segment $S=(b^-, b^+)$ and open interval $J\subseteq S$, the \emph{left-degree} of $J$ is the number of pillars $P$ with base $b$ such that $b \le b^-$, and there exist intersecting L-shapes $L,L'\in {\cal L}$ such that $L$ is assigned to pillar $P$ and $p(L')\subseteq J$. We denote the left-degree of $J$ by $d_{({\cal P},\prec)}^\ell(J)$. The \emph{left $({\cal P},\prec)$-neighborhood} $N_{({\cal P},\prec)}^\ell(J)$ is the set of all such pillars $P$. Throughout the paper (beginning now) we omit subscripts $({\cal P},\prec)$ when the ordered pillars are clear from context. 

In a similar manner, the \emph{additional-degree} of $J$ is the number of pillars $P$ with base $b$ such that $b \le b^-$, and there exist intersecting L-shapes $L,L'\in {\cal L}$ such that $L$ is assigned to pillar $P$, and $\ell(L')\in J$, and $r(L')\in S$, and $p(L')\setminus J$ does not contain the left endpoint of any L-shape in $\cal L$. We denote the additional-degree by $d^{a}(J)$, and the \emph{additional-neighborhood} $N^a(J)$ is the set of all such pillars $P$. Notice that $N^\ell (J) \subseteq N^a (J)$. It is the elements of $N^a(J) \backslash N^\ell (J)$ that will act as the aforementioned buffer for the left-degree.

The \emph{right-degree} of $J$ is the number of pillars $P$ with base $b$ such that $b \ge b^+$, and there exists an L-shape $L\in {\cal L}$ such that $L$ is assigned to pillar $P$ and $\ell(L)\in J$. Similarly we denote the right-degree of $J$ by $d^r(J)$, and the \emph{right-neighborhood} $N^r(J)$ is the set of all such pillars $P$. Finally, the \textit{degree} of $J$, denoted $d(J)$, is equal to the sum of the left- and right-degrees of $J$.

\input ./figures/degrees.tex

We require a lemma which was used to prove the quadratic $\chi$-bounding function for overlap graphs~\cite{DaviesMcCarty}.
The lemma is stated somewhat differently in this paper, so we include a proof for completeness. As hinted at earlier, the lemma can be seen as a permutation graph analogue of a theorem of Capoyleas and Pach~\cite{capoyleasturan62}.
A somewhat more abstract version of this lemma with tight bounds was also proven in~\cite{davies2021circle}.

\begin{lemma}[Davies and McCarty \cite{DaviesMcCarty}]\label{permutation bound}
    Let $\cal I$ be a finite set of open intervals with non-empty intersection. If $A$ is the set of endpoints of intervals in $\cal I$ and $\omega$ is the clique number of the overlap graph of $\cal I$, then $|{\cal I}|\le \omega (|A|-1)$.
\end{lemma}
\begin{proof}
    Let $c \in \mathbb{R}$ be a point which is contained in every interval in $\cal I$. The overlap graph of $\cal I$ is a permutation graph where the two permutations are given by 1) the order of the leftmost endpoints of intervals in $\cal I$ (which are all less than $c$) and 2) the order of the rightmost endpoints of intervals in $\cal I$ (which are all more than $c$). So as permutation graphs are perfect, there is a partition of $\cal I$ into $\omega$ sets, each containing no pair of overlapping intervals. Hence it is enough to prove the lemma in the case that $\omega=1$.
	
	This is trivial if $|{\cal I}|=1$, so we may proceed inductively. It is enough to show that there exists a point in $A$ which is an endpoint of just one interval in ${\cal I}$. Let $a^-= \max\{a\in A : a< c\}$ and $a^+=\min\{a\in A : a > c\}$. Suppose that each of $a^-$ and $a^+$ are endpoints of at least two intervals in ${\cal I}$. Then there exists points $b^+,b^-\in A$ with $b^+ > a^+$ and $b^- < a^-$ so that $(a^-,b^+)$ and $(b^-,a^+)$ are intervals in $\cal I$. But then $(a^-,b^+)$ and $(b^-,a^+)$ overlap, a contradiction. So we conclude that either $a^-$ or $a^+$ is an endpoint of just one interval in ${\cal I}$, as we require.
\end{proof}

We now prove an extremal lemma for the additional- and left-degrees, and then we conclude this section by proving an extremal lemma for the right-degree.

\begin{lemma}\label{left}
	Let $\cal L$ be a flat L-collection with ordered pillars $({\cal P}, \prec)$ so that $G({\cal L})$ has clique number $\omega$, and let $\cal J$ be a finite, non-empty collection of disjoint open intervals contained in some segment $S$. Then $\sum_{J \in {\cal J}}d^a(J) \le \omega (|{\cal J}| + d^\ell(S) -1)$.
\end{lemma}

\begin{proof}
    Going for a contradiction, suppose otherwise. Construct an auxiliary collection of open intervals $\cal I$ where, for each $J \in {\cal J}$ and $P \in N^a(J)$, we add an interval between the rightmost endpoint of $J$ and the base of $P$. We may assume that $\cal I$ is non-empty since otherwise the lemma is trivially true. Thus $\cal I$ has non-empty intersection (note that every interval in $\cal I$ contains the leftmost interval in $\cal J$). Furthermore, $|{\cal I}| = \sum_{J \in {\cal J}}d^a(J)$ and the intervals in $\cal I$ have at most $|{\cal J}| + |\bigcup_{J \in {\cal J}}N^a(J)|$ different endpoints. By the definition of the additional- and left-degrees, $|\bigcup_{J \in {\cal J}}N^a(J)| \le d^\ell(S)$. So by Lemma~\ref{permutation bound}, the overlap graph of $\cal I$ has clique number greater than~$\omega$.
     
    It follows that there are intervals $J_1, \ldots, J_{\omega+1} \in {\cal J}$ and pillars $P_1 \in N^a(J_1), \ldots,\allowbreak P_{\omega+1}\in N^a(J_{\omega+1})$ so that, where $b_1, \ldots, b_{\omega+1}$ are the bases of $P_1, \ldots, P_{\omega+1}$ respectively and $j_1, \ldots, j_{\omega+1}$ are the rightmost endpoints of $J_1, \ldots, J_{\omega+1}$ respectively, we have that $b_1 <\cdots< b_{\omega+1} < j_1 < \cdots < j_{\omega+1}$. Therefore there exist L-shapes $L_1,\dots,L_{\omega +1}, L_1', \dots L_{\omega+1}'\in {\cal L}$ such that for each $i\in \{1,\dots, \omega +1\}$, $L_i$ is assigned to $P_i$ and intersects $L_i'$, and $\ell(L_i') \in J_i$, and $p(L_i')\setminus J_i$ does not contain the left endpoint of any L-shape in $\cal L$.
    
    Notice that $P_1 \prec \dots \prec P_{\omega+1}$; this follows from Lemma~\ref{lem:assigned} since $b_i, b_{i+1}, \ldots, b_{\omega+1} \in p(L_i)$ for each $i \in \{1,2,\ldots, \omega+1\}$. Hence, again by Lemma~\ref{lem:assigned}, $\ell(L_1) \in (-\infty, b_1)$ and for $i >1$, $\ell(L_i) \in (b_{i-1}, b_i)$. So in particular $\ell(L_1)< \cdots< \ell(L_{\omega+1})$. As the clique number of the intersection graph of $\cal L$ is at most $\omega$, there exist some $i<j$ such that $L_i$ does not intersect $L_{j}$. So $h(L_i)>h(L_j)$. But then as $L_j$ intersects $L_j'$, and $L_i$ intersects $L_i'$, and $\cal L$ is flat, it follows that $\ell(L_j') \in p(L_i')\setminus J_i$, a contradiction.
\end{proof}

\input ./figures/cascading.tex

We now wish to obtain a similar lemma for the right-degree. This one requires an extra factor of $\omega$. In order to prove the lemma for the right degree, it is helpful to introduce another configuration that yields a clique elsewhere. Let $\cal L$ be a flat L-collection with ordered pillars $({\cal P}, \prec)$. We say that L-shapes $L_1,\dots , L_{t}$ are \emph{cascading} if there exist pillars $P_1,\dots, P_{t}$ with bases $b_1,\dots , b_{t}$ respectively such that
\begin{itemize}\itemsep0em
    \item[1)] $\ell(L_1)<\dots < \ell(L_{t}) <b_1 < \dots < b_{t}$,
    \item[2)] $h(L_1)>\dots > h(L_{t})$ and $P_1 \succ \dots \succ P_{t}$, and
	\item[3)] $L_{t}$ is a support of $P_{t}$ and for each $i<t$, $L_i$ is assigned to $P_i$.
\end{itemize}
\noindent We purposefully do not require that $L_t$ is assigned to $P_t$.
See Figure~\ref{fig:cascad} for an example of cascading L-shapes, as well as one way that they may yield a clique.

Next we prove that cascading L-shapes $L_1,\dots , L_{t}$  yield a clique of size $t$. It is convenient for the sake of induction to prove the following stronger statement.

\begin{lemma}\label{clique}
	Let $\cal L$ be a flat L-collection with ordered pillars $({\cal P},\prec )$ and cascading L-shapes $L_1,\dots, L_{t}$. Then there is a clique $\{L_1^*, \ldots, L_t^*\}$ in the intersection graph $G({\cal L})$ so that $\ell(L_1^*) < \cdots < \ell(L_t^*)$, for each $i\le t$, $L^*_i$ supports $P_i$, and $L_t^* = L_t$.
\end{lemma}
\begin{proof}
    We prove the claim by induction on $t$. The base case of $t=1$ holds trivially. So, let $P_1,\dots , P_{t}$ be the pillars corresponding to $L_1,\dots, L_{t}$, with bases $b_1,\dots , b_{t}$ respectively. 
    
    By condition 3), $b_t \in p(L_t)$; so $b_1, \ldots, b_t \in p(L_t)$ as well. By condition 2), the pillar $P_{t-1}$ must intersect $L_t$. In fact, as $L_t$ supports $P_t$ and $P_t\prec P_{t-1}$, the pillar $P_{t-1}$ first intersects $L_t$ at a point on the vertical segment of $L_t$. So there is a support $L_{t-1}^*$ of $P_{t-1}$ with $\ell(L_{t-1}^*)<\ell(L_t)$ so that $L_{t-1}^*$ and $L_t$ intersect. If $t=2$ we are done; otherwise, by induction on $t$, it suffices to prove that $\ell(L_{t-2})<\ell(L_{t-1}^*)$, as then $L_1, \ldots, L_{t-2}, L_{t-1}^*$ would be cascading. Otherwise, $\ell(L_{t-1}^*) < \ell(L_{t-2}) <\ell(L_t)$ and so $L_{t-1}^*$ and $L_{t-2}$ intersect, as $h(L_{t-1}^*)<h(L_{t-1})<h(L_{t-2})$. But then by Lemma~\ref{lem:assigned} some pillar $P\prec P_{t-1}$ must intersect $L_{t-2}$ as $L_{t-1}^*$ supports $P_{t-1}$. This contradicts the fact that $L_{t-2}$ is assigned to $P_{t-2}$.
\end{proof}

We complete this section by proving the extremal lemma for the right-degree.

\begin{lemma}\label{right}
	Let $\cal L$ be a flat L-collection with ordered pillars $(P, \prec)$ so that $G({\cal L})$ has clique number $\omega$, and let $\cal J$ be a finite, non-empty collection of disjoint open intervals contained in some segment $S$. Then $\sum_{J \in {\cal J}}d^r(J) \le 2\omega(2\omega -1) (|{\cal J}| + d^r(S) -1)$.
\end{lemma}

\begin{proof}
For the sake of contradiction, suppose otherwise. The first part of the proof is very similar to the proof of Lemma~\ref{left}, where we consider an auxiliary collection of open intervals. 

So, construct a collection of open intervals $\cal I$ where, for each $J \in {\cal J}$ and $P \in N^r(J)$, we add an interval between the leftmost endpoint of $J$ and the base of $P$. We can assume that $\cal I$ is non-empty since otherwise the lemma holds. Then $\cal I$ has non-empty intersection since every interval in $\cal J$ contains the rightmost interval in $\cal J$. Furthermore, $|{\cal I}|=d^r({\cal J})$ and the intervals in $\cal I$ have at most $|{\cal J}| + |\bigcup_{J \in {\cal J}}N^r(J)| \leq |{\cal J}|+d^r(S)$ different endpoints. So by Lemma~\ref{permutation bound}, the overlap graph of $\cal I$ has clique number greater than~$2\omega(2\omega-1)$. 

Hence there are L-shapes $L_1,\dots, L_{2\omega(2\omega -1) +1}$ which are assigned to pillars $P_1, \ldots,\allowbreak P_{2\omega(2\omega -1) +1}$ with bases $b_1, \ldots, b_{2\omega(2\omega -1) +1}$ respectively, so that \begin{align*}\ell(L_1)<\cdots < \ell(L_{2\omega(2\omega -1) +1})<b_1< \cdots <b_{2\omega(2\omega -1) +1}.\end{align*}

\noindent We claim that $P_1 \succ \dots \succ P_{2\omega(2\omega -1)+1}$ as well. Otherwise, if $P_i \prec P_j$ for some $i<j$, then $L_j$ intersects $L_i$, as $L_j$ must intersect $P_j$ and not $P_i$. Furthermore as $\cal L$ is flat, we have $r(L_j)>r(L_i)$. But then $[\ell(L_j),r(L_i)] \times [0,h(L_i)] \subset p(L_j) \times [0,h(L_j)]$ must contain a point of $P_i$, contradicting Lemma~\ref{lem:assigned}.

Now, by the Erd\H{o}s–Szekeres Theorem~\cite{erdos1935combinatorial}, the sequence $h(L_1), \ldots, h(L_{2\omega(2\omega -1) +1})$ contains either a decreasing subsequence of size $2\omega+1$ or an increasing subsequence of size $2\omega$. We take care of the two cases separately. 
\medskip

\textbf{Case 1:} There are indices $i_1 < \cdots<i_{2\omega+1}$ so that $h(L_{i_1})> \cdots> h(L_{i_{2\omega+1}})$.

Note that we already have conditions 1) and 2) for the L-shapes $L_{i_1}, \ldots, L_{i_{2\omega+1}}$ to be cascading. So, roughly, it just remains to find a special L-shape $L'$ which supports its pillar. First, if $b_{i_{\omega+1}} \notin p(L_{i_{2\omega+1}})$, then there must exist an L-shape $L'$ supporting $P_{i_{2\omega+1}}$ with $\ell(L_{i_{2\omega+1}}) < \ell(L')< b_{i_{\omega+1}}$ and $h(L') < h(L_{i_{2\omega + 1}})$. But then the L-shapes $L_{i_{\omega+1}},\dots , L_{i_{2\omega}},L'$ are cascading, and by Lemma~\ref{clique} there is a clique of size $\omega+1$, a contradiction. 

Hence $b_{i_{\omega+1}} \in p(L_{i_{2\omega+1}})$. Now there must exist some L-shape $L'$ supporting $P_{i_{\omega+1}}$ with $\ell(L') \le \ell(L_{i_{2\omega+1}})$ and $h(L') \le h(L_{i_{2\omega+1}})$ (possibly with $L'=L_{i_{2\omega+1}}$). We must also have that $\ell(L_\omega)<\ell(L')$ as otherwise $L_\omega$ and $L'$ intersect and, since $\cal L$ is flat, $b_{i_{\omega+1}} \in p(L_{\omega})$, a contradiction to the fact that $L_\omega$ is assigned to $P_\omega$. But now the L-shapes $L_{i_1}, \dots , L_{i_{\omega}}, L'$ are cascading, again contradicting Lemma~\ref{clique}.
\medskip

\textbf{Case 2:} There are indices $j_1 < \cdots<j_{2\omega}$ so that $h(L_{j_1})< \cdots< h(L_{j_{2\omega}})$.

Now there must exist some $t\in \{1,\dots , \omega\}$ such that $\ell(L_{j_{\omega +1}}) \notin p(L_{j_t})$, as otherwise $L_{j_1}, \dots , L_{j_{\omega +1}}$ would be pairwise intersecting. As $L_{j_t}$ intersects the pillar $P_{j_t}$, there must exist some L-shape $L'$ with $\ell(L')< \ell(L_{j_{\omega +1}})$ and $h(L')< h(L_{j_{\omega +1}})$ that supports $P_{j_t}$. But now as $\cal L$ is flat, the L-shapes $L', L_{j_{\omega +1}}, \dots , L_{j_{2\omega}}$ are pairwise intersecting, a contradiction. This completes the proof of Lemma~\ref{right}.
\end{proof}

\section{Main result}
Theorem~\ref{thm:main} will follow quickly from the main proposition below on finding complete pillar assignments. Recall that the degree is the sum of the left- and right-degrees.

\begin{proposition}\label{complete pillar assignment}
	Every flat L-collection $\cal L$ whose intersection graph has clique number $\omega$ has a complete pillar assignment using at most $4\omega^2 -\omega +2\lceil 4\log_2 (\omega) \rceil +11$ colors.
\end{proposition}

\begin{proof}
	This is certainly true when $\omega=1$ because, by Lemma~\ref{lem:assigned}, we can continue placing bases until every L-shape is assigned (we give every pillar the same color). So we may assume that $\omega \ge 2$. 
	
	Throughout the proof we will only use colors in $\{1,\dots ,4\omega^2 -\omega +2\lceil 4\log_2 (\omega) \rceil +11\}$. Now, choose a pillar assignment $({\cal P}, \prec ,c)$ so that \begin{itemize}\itemsep0em
    \item[1)] each segment has degree at most $4\omega^2 -\omega +\lceil 4\log_2 (\omega) \rceil +6$, and
    \item[2)] subject to the above, $\phi_{({\cal P}, \prec ,c)}$ colors as many L-shapes in $\cal L$ as possible.
    \end{itemize}
	\noindent Suppose for the sake of contraction that $({\cal P}, \prec ,c)$ is not complete. Then there is an L-shape $L^* \in {\cal L}$ which is not assigned to any pillar; let $S=(b^-,b^+)$ be the segment containing $\ell(L^*)$. Also, let $B$ be the set of bases of pillars in $\cal P$.
	
	We have the following key claim. Note that in both the statement of the claim and its proof, all degrees are with respect to the ordered pillars $({\cal P}, \prec)$.
	
	\begin{claim}
	\label{claim:B*}
		There exists a set $B^*\subset S$ of at most $32\omega^4 -1$ bases so that $p(L^*)\cap B^*$ is non-empty and each segment of $B \cup B^*$ which is contained in $S$ has degree at most $4\omega^2-\omega +1$.
	\end{claim}

\begin{proof}
	Choose bases $b^-=b_0 < b_1 < \dots < b_t < b^+$ such that for each $i\in \{1,\dots ,t\}$, the degree of $(b_{i-1},b_i)$ is at most $4\omega^2 -\omega +1$, and either $d^a(b_{i-1},b_i) \geq \omega+1$ or $d^r(b_{i-1},b_i) \geq 4\omega^2 -2\omega +1$, and subject to this, $t\ge 0$ is maximized. A maximum $t$ exists because the second condition ensures that there is an L-shape with left endpoint in $(b_{i-1}, b_i)$, and $\cal L$ is finite.
	
	We claim that the degree of $(b_{t},b^+)$ is at most $4\omega^2 -\omega +1$. Otherwise we will place another base $b_{t+1}$ between $b_t$ and $b^+$; consider starting with $b_{t+1}$ just slightly past $b_t$ and moving $b_{t+1}$ past a single point in $\{\ell(L):L \in {\cal L}\}\cup \{r(L):L \in {\cal L}\}$ at a time. If $b_{t+1}$ is moved past a left endpoint, then the left-degree does not change and the right-degree goes up by at most $1$. Otherwise, $b_{t+1}$ is moved past some right endpoint to a new location $b_{t+1}'$. Then the right-degree does not change and the left-degree goes up by at most $d^a(b_t, b_{t+1})-d^\ell(b_t, b_{t+1})$ because $N^\ell(b_t, b_{t+1}') \subseteq N^a(b_t,b_{t+1})$. It follows that if $b_{t+1}'$ is the first location where the degree is more than $4\omega^2 -\omega +1$, then the previous location $b_{t+1}$ gives a contradiction to the choice of $t$.
	
	Now, let ${\cal J}_a$ (resp. ${\cal J}_r$) be the set of segments of $B \cup \{b_1, \ldots, b_t\}$ which are contained in $S$ and have additional-degree at least $\omega+1$ (resp. right-degree at least $4\omega^2-2\omega+1$). By definition we have $|{\cal J}_a|+ |{\cal J}_r| \ge t$. By Lemmas~\ref{left} and~\ref{right},\begin{align*}
	|{\cal J}_a|(\omega+1) &\le \sum_{J \in {\cal J}_a}d^a(J) < \omega(|{\cal J}_a|+ 4\omega^2-\omega + \lceil 4\log_2 (\omega) \rceil +6) \textrm{ and}\\
	|{\cal J}_r|(4\omega^2-2\omega+1) &= \sum_{J \in {\cal J}_r}d^r(J) < (4\omega^2-2\omega)(|{\cal J}_r|+ 4\omega^2-\omega +\lceil 4\log_2 (\omega) \rceil +6).
	\end{align*}
	As a consequence, \begin{align*}
	|{\cal J}_a| &< \omega\left(4\omega^2-\omega +\lceil 4\log_2 (\omega) \rceil +6\right) \textrm{ and}\\
	|{\cal J}_r| &< (4\omega^2-2\omega)(4\omega^2-\omega +\lceil 4\log_2 (\omega) \rceil +6).
	\end{align*}
	
	As both sides of the inequalities are integers and $t \leq |{\cal J}_a|+|{\cal J}_r|$, summing the right sides and then using the fact that $\omega \geq 2$ we obtain\begin{align*}
	t+2 &\leq (4\omega^2-\omega)(4\omega^2-\omega +\lceil 4\log_2 (\omega) \rceil +6)\\
	&\leq 4\omega^2(4\omega^2+3\omega +6)\\
	&\leq 32\omega^4.
	\end{align*}
	Finally, let $b\in p(L^*)$; then $B^*=\{b,b_1,b_2,\dots ,b_t\}$ provides the desired subset of~$S$.
\end{proof}

Now fix $B^* \subset S$ as in the claim, with $|B^*|\leq 32\omega^4-1$. Observe that, by the choice of $({\cal P}, \prec, c)$, there are $\lceil 4\log_2 (\omega) \rceil+5=\lceil \log_2 (32\omega^4) \rceil$ colors available which are not used to color any pillar in $N^\ell(S)\cup N^r(S)$ (where neighborhoods are with respect to $({\cal P}, \prec)$). So by Lemma~\ref{divide-and-conquer} (the divide-and-conquer lemma), there is a total ordering $\prec^*$ and a coloring $c^*$ of $B^*$ so that $c^*$ uses only these available colors and, for all $b, b' \in B^*$ with $b <b'$ and $c^*(b)=c^*(b')$, there exists $b^* \in (b, b')\cap B^*$ so that $b^* \prec^* b$ and $b^* \prec^* b'$. 

Now, beginning with $({\cal P}, \prec, c)$, we place additional pillars ${\cal P}^*$ at bases in $B^*$ according to the ordering $\prec^*$ and color them according to $c^*$ in order to obtain a new pillar assignment ${\cal A}^*=({\cal P} \cup {\cal P}^*, \prec \cup \prec^*, c \cup c^*)$ (note that we also view $c^*$ as a coloring of ${\cal P}^*$ and that each pillar in ${\cal P}$ precedes each pillar in ${\cal P}^*$). We will show that ${\cal A}^*$ is a pillar assignment which contradicts the choice of $({\cal P}, \prec, c)$. 

\begin{claim}
\label{claim:uncoloured}
If $L \in {\cal L}$ is colored by $\phi_{{\cal A}^*}$ but not $\phi_{({\cal P}, \prec, c)}$, then $p(L) \subset S$.
\end{claim}
\begin{proof}
By Lemma~\ref{lem:assigned} we know that $p(L)$ is contained in a segment of $B$. So if $p(L)$ is not contained in $S = (b^-, b^+)$, then $\ell(L)<b^-$ and $L$ would have been colored by either the pillar with base $b^-$ or an earlier pillar, a contradiction.
\end{proof}

The next two claims will complete the proof.
\begin{claim}
\label{claim:pillar}
The tuple ${\cal A}^*$ is a pillar assignment so that $\phi_{{\cal A}^*}$ colors more L-shapes than $\phi_{({\cal P}, \prec, c)}$.
\end{claim}
\begin{proof}
The latter statement will be easy as $L^*$ is not colored by $\phi_{({\cal P}, \prec, c)}$ but is colored by $\phi_{{\cal A}^*}$ (since $p(L^*)\cap B^*$ is non-empty). So it just remains to show that ${\cal A}^*$ is a pillar assignment. Going for a contradiction, suppose that $L_1$ and $L_2$ are intersecting L-shapes assigned respectively to pillars $P_1$ and $P_2$ of the same color. 

Then one of the L-shapes, say $L_1$, is not colored by $\phi_{({\cal P}, \prec, c)}$. So $p(L_1) \subset S$ by Claim~\ref{claim:uncoloured}. Then by the choice of $c^*$, the other L-shape $L_2$ must not be colored by $\phi_{({\cal P}, \prec, c)}$ either. Thus, there must be a base $b^* \in B^*$ which is between the bases of $P_1$ and $P_2$ and earlier than both according to $\prec^*$. But then either $L_1$ or $L_2$ would have been colored by the pillar with base $b^*$ or an earlier pillar, a contradiction.
\end{proof}

\begin{claim}
The degree of each segment of ${\cal A}^*$ with respect to $({\cal P} \cup {\cal P}^*, \prec \cup \prec^*)$ is at most $4\omega^2 -\omega +\lceil 4\log_2 (\omega) \rceil +6$.
\end{claim}
\begin{proof}
Let $S^*$ be a segment of ${\cal A}^*$. If $S^*$ is disjoint from $S$ then its degree does not change, so we may assume that $S^*\subset S$. Now consider a pillar $P$ with base $b$ which contributes to the degree of $S^*$ in $({\cal P} \cup {\cal P}^*, \prec \cup \prec^*)$. If $b\in B$ then $P$ contributes to the degree of $S^*$ in $({\cal P}, \prec)$ as well; by Claim~\ref{claim:B*} there are at most $4\omega^2-\omega+1$ such pillars $P$. If $b\in B^*$, then there cannot be any earlier base in $\prec^*$ which is between $b$ and the endpoint of $S^*$ closer to $b$; so so there is at most one such pillar $P$ for each color of $c^*$. In total the new degree is at most\[(4\omega^2 -\omega +1) + (\lceil 4\log_2(\omega) \rceil+5)  = 4\omega^2 -\omega +\lceil 4\log_2 (\omega) \rceil +6,\]which completes the proof of the claim.
\end{proof}

By the above two claims, ${\cal A}^*$ is a pillar assignment which contradicts the choice of $({\cal P}, \prec, c)$. This therefore completes the proof of Proposition~\ref{complete pillar assignment}.
\end{proof}

Theorem~1 now quickly follows.

\begin{proof}[Proof of Theorem~1]
We may assume that $\omega \geq 2$ as otherwise the theorem is trivially true. By combining Proposition~\ref{complete pillar assignment} with Lemmas~\ref{flattened} and~\ref{colouring L with pillars}, we obtain that the grounded L-graph has chromatic number at most $(2\omega^2)(4\omega^2 -\omega +2\lceil 4\log_2 (\omega) \rceil +11)$. So the theorem holds if $\omega =2$. If $\omega \geq 3$ then\begin{align*}
    (2\omega^2)(4\omega^2 -\omega +2\lceil 4\log_2 (\omega) \rceil +11) &\leq (2\omega^2)(4\omega^2 +7\omega +11)\\
    & \leq (2\omega^4)\left(4+\frac{7}{3}+\frac{11}{9} \right)\\
    &\leq 16\omega^4
\end{align*}
and still the theorem holds.
\end{proof}

\section{Separation results}

The goal of this section is to prove the separation results of Theorem \ref{thm:separation}.
For all of our constructions, we would like to force that some objects are attached to the grounding line in some specific 
order.
Such an approach was already used by Cardinal et al.~\cite{CardinalFMTV18}, 
who observed that a simple family of grounded strings representing a cycle has a fairly rigid structure, 
which yields a specific order of the base points of strings from this family and the base points of strings representing vertices ``appropriately'' adjacent to the cycle.
This approach was formalized in so-called Cycle Lemma (Lemma~6 of \cite{CardinalFMTV18}), which was proved for various subclasses of outer-1-string graphs.
We also note that the separation results obtained by Jel{\'{\i}}nek and T{\"{o}}pfer~\cite{JelTop19} were based on similar ideas.
In our work, we use a weaker version of Cycle Lemma for outer-1-string graphs, sufficient for our applications.
Moreover, it turns out that in the class of interval filament graphs, a cycle also has a fairly rigid representation, 
which allows us to control the ranges of the domains representing the vertices from the cycle, 
as well as vertices ``appropriately'' adjacent to the cycle.
The above ideas are brought to life in so called Cycle Lemma for Interval Filament Graphs. 

Before we prove Cycle Lemma for Interval Filament Graphs, we need some 
preparation first.
Suppose $G$ is an intersection graph of a family of interval filaments $\mathcal{F}$. 
Given a vertex $v$ in $G$, by $r_{\mathcal{F}}(v)$ we denote the interval filament in $\mathcal{F}$ representing $v$ 
and by $\dom_{\mathcal{F}}(v)$ we denote the domain of $r_{\mathcal{F}}(v)$.
Usually we omit the subscript if the family $\mathcal{F}$ is clear from the context.
Given a path $P$ in $G$, we say that $P$ is represented in $\mathcal{F}$ by
a \emph{chain of overlapping interval filaments} if the consecutive vertices on $P$ can be labeled by $v_1,\ldots,v_k$ so as the following holds:
\begin{itemize}
 \item for every $i \in [k-2]$ we have $\dom(v_i) < \dom(v_{i+2})$, 
 \item for every $i \in [k-1]$ the intervals $\dom(v_i)$ and $\dom(v_{i+1})$ overlap.
\end{itemize}
See Figure~\ref{fig:cycle_lemma_for_interval_filaments_1} for an illustration.
\input ./figures/cycle_lemma_for_interval_filaments_1.tex

\begin{lemma}[Cycle Lemma for Interval Filament Graphs]
Let $G$ be an intersection graph of a family of interval filaments $\mathcal{F}$.
Suppose $C \subset V(G)$ induces a cycle of size~$n$ in $G$ for some $n \geq 7$.
Then, there is a path $P$ of size $n-4$ in $C$ such that $P$ is
represented in $\mathcal{F}$ by a chain of overlapping interval filaments (see Figure~\ref{fig:cycle_lemma_for_interval_filaments_1}).

Suppose $k = n-4$ and $v_1 \ldots v_k$ are the consecutive vertices of $P$ enumerated so that $\dom(v_i)$ is to the left of $\dom(v_j)$ for $i < j$. 
If $v$ is a vertex in $V \setminus C$ such that
$v$ is adjacent to the vertices $v_i,v_j$ for some $i < j$ in $[2,k-1]$ and all the neighbors of $v$ from $C$ are contained in the set $\{v_{i},v_{i+1},\ldots,v_j\}$,
then the left endpoint of $\dom(v)$ is contained in $\dom(v_i)$ and the right endpoint of $\dom(v)$ is contained in $\dom(v_j)$.
\end{lemma}

\begin{proof}
See Figure~\ref{fig:cycle_lemma_for_interval_filaments_1} for an illustration.
Suppose $a,b,c$ are three independent vertices in $C$.
Consider the intervals $\dom(a), \dom(b), \dom(c)$.
We claim that they are pairwise disjoint or one among them contains the remaining two, which are disjoint.
Suppose otherwise.
Without loss of generality we may assume that $\dom(a)$ is contained in $\dom(b)$ and that either $\dom(b)$ is contained in $\dom(c)$ or $\dom(b)$ and $\dom(c)$
are disjoint.
In any case, note that some vertex on the path in $C$ that joins $a$ and $c$ and avoids $b$ 
is represented by an interval filament that intersects $r(b)$, which can not be the case.
Next, we claim that there are two non-adjacent vertices $w,u$ in $C$ such that $\dom(u)$ is contained in $\dom(w)$.
Assume otherwise.
Then, our first claim asserts that the domains of every three independent vertices in $C$ must be pairwise
disjoint.
In particular, it follows that the domains of every two adjacent vertices from $C$ must overlap.
However, such a set of interval filaments would represent a graph that consists of some number of paths, which can not be the case.
So, suppose $u,w \in C$ are such that $\dom(u)$ is contained in $\dom(w)$.
Let $P'$ be the set that contains all the vertices from~$C$ except from $w$ and its two neighbours, $w'$ and $w''$.
Clearly, $P'$ induces a path in $G$ of size $n-3$.
We denote the consecutive vertices of $P'$ by $w_1,\ldots,w_{n-3}$ such that $w'w_1$ and $w_{n-3}w''$ are edges of $C$.
Since $w$ is adjacent to no vertex of $P'$, the domains of the interval filaments representing the vertices in $P'$ are contained in $\dom(w)$.
Using similar arguments as earlier, we show that the domains of every three independent vertices from $P'$ are pairwise disjoint.
It follows that every two inner vertices in $P'$ that are adjacent must overlap.
However, it could happen that $\dom(w_1) \subset \dom(w_2)$ or $\dom(w_{n-3}) \subset \dom(w_{n-4})$ -- see Figure~\ref{fig:cycle_lemma_for_interval_filaments_2}.
Note that we can not have $\dom(w_1) \subset \dom(w_2)$ and $\dom(w_{n-3}) \subset \dom(w_{n-4})$ at the same time.
Indeed, assuming otherwise, the domains of $w'$ and $w''$ would contain the interval $\bigcup_{i \in [n-3]} \dom(w_i)$.
It follows that the intervals $\dom(w'')$ and $\dom(w')$ are nested.
If $\dom(w'') \subset \dom(w')$ then $r(w_{1})$ intersects $r(w'')$ as $r(w_{1})$ intersects $r(w')$ and $r(w'')$ and $r(w')$ are disjoint, 
which can not be the case.
See Figure~\ref{fig:cycle_lemma_for_interval_filaments_2} for an illustration.
We obtain an analogous contradiction if $\dom(w') \subset \dom(w'')$.
Consequently, we note that $P' - \{w_1\}$ or $P' - \{w_{n-3}\}$ is a path of size $n-4$ represented by a chain of overlapping interval filaments.  

\input ./figures/cycle_lemma_for_interval_filaments_2.tex

To prove the second part of the claim, let $L$ and $R$ denote the right endpoint of $\dom(v_1)$ and the left endpoint of 
$\dom(v_k)$, respectively.
Note that $L < R$ as $k \geq 3$.
Let $J$ be any curve from $L$ to $R$ that is contained in the interval filaments representing the path $P''$ of $C$ from $v_1$ to $v_k$ avoiding $v_{2},\ldots,v_{k-1}$.
See Figure~\ref{fig:cycle_lemma_for_interval_filaments_1} for an illustration.
Note that $r(v)$ can not intersect $J$ as $v$ is not adjacent to the vertices of $P''$.
In particular, it follows that $\dom(v)$ is contained in the interval $(L,R)$.
Now the thesis follows easily from the fact that the vertices $v_1\ldots v_k$ are represented by the chain of overlapping intervals. 
\end{proof}

Suppose that $G$ is the intersection graph of a simple family of grounded curves $\mathcal{F}$.
Given a vertex $v \in G$, by $s_{\mathcal{F}}(v)$ we denote the string representing the vertex $v$ and by $b_{\mathcal{F}}(v)$
we denote the base of $s_\mathcal{F}(v)$.
Again, we omit the subscript if $\mathcal{F}$ is clear from the context.
The next lemma is a weaker version of Cycle Lemma proved in \cite{CardinalFMTV18}.
For the sake of completeness, we enclose its proof, which follows the same reasoning as the proof in~\cite{CardinalFMTV18}.

\begin{lemma}[Cycle Lemma for Outer-1-string Graphs]
Suppose $G$ is the intersection graph of a simple family of grounded curves $\mathcal{F}$.
Let $I \subset C \subset V(G)$ be such that $G[C]$ induces a cycle in $G$ and $I$ is an independent set in $C$. 
Then the cyclic order of the vertices from $I$ given by their occurrence on $C$ coincides, up to reversal,
to the cyclic order of their base points on the grounding line, where here we conveniently assume that the leftmost base point follows
the rightmost base point.
\end{lemma}

\begin{proof}
Denote the vertices of $I$ by $v_0,\ldots,v_{n-1}$ so that $b(v_i) < b(v_j)$ holds for every $i < j$ in $[n-1]$.
Every string $s(v_i)$ is crossed by two strings $s(v_{i-1})$ and $s(v_{i+1})$ (we take indices modulo $n$):
the part of $s(v_i)$ between $b(v_i)$ and the first intersection is called the \emph{initial part} of $s(v_i)$, denoted by $i(s_i)$, 
and the part between the two intersections is called the \emph{central part} of $s(v_i)$, denoted by $c(v_i)$.
Note that the union of all the central parts of the vertices of $C$ forms a Jordan curve, denoted by $J(C)$.
Note also that the interiors of the initial parts $i(v_i)$ for $i \in [n-1]$ as well as $J(C)$ are pairwise disjoint.
Now, the ends of $i(v_0)$ and $i(v_{n-1})$ on $J(C)$ partition $J(C)$ into two parts, and exactly one of these parts, 
say $J'(C)$, contains all the central parts of the vertices from $I - \{v_0,v_{n-1}\}$.
Indeed, assuming otherwise, some vertex from $I \setminus \{v_0,v_{n-1}\}$ would have its base outside the interval $[b(v_{0}),b(v_{n-1})]$.
Now, we easily observe that when we traverse $J'(C)$ from the endpoint of $i(v_{0})$ to the endpoint of $i(v_{n-1})$, 
we encounter the central parts of the vertices from $I \setminus \{v_{0},v_{n-1}\}$ in the order $v_1,\ldots,v_{n-2}$.
Now, the lemma follows as the central parts of the vertices from $C$ occur in $J(C)$ in the order corresponding to their occurrence on $C$.
\end{proof}

Now we are ready to prove the separation theorem.

\begin{proof}[Proof of Theorem \ref{thm:separation}]
We recall that we need to prove the following:
\begin{enumerate}
 \item~\label{enum:interval_filament_not_polygon_circle} There is an interval filament graph which is not a polygon-circle graph.
 \item~\label{enum:flat_grounded_L_graph_not_interval_filament} There is a (flat) grounded L-graph which is not an interval filament graph.
 \item~\label{enum:monotone_L_graph_not_interval_filament} There is a monotone L-graph which is not an interval filament graph.
 \item~\label{enum:polygon_circle_not_outer_1_string} There is a polygon-circle graph which is not an outer-$1$-string graph.
\end{enumerate}
\noindent First we prove statement \eqref{enum:interval_filament_not_polygon_circle}.
Consider the graph $G$ that consists of: 
\begin{itemize}
 \item a cycle $C$, whose consecutive vertices are denoted by $v^1_0,\ldots,v^1_{12}$, $v^2_0,\ldots,v^2_{12}$,
 \item four vertices $a^1,b^1,a^2,b^2$, where $a^i$ is adjacent to $b^i,v^i_1,v^i_7$, 
 and $b^i$ is adjacent to $a^i,v^i_3,v^i_5$, for $i \in [2]$.
\end{itemize}
Clearly, $G$ is an intersection graph of interval filaments, as shown in Figure~\ref{fig:interval_filament_not_polygon_circle} to the left.
\input ./figures/interval_filament_not_polygon_circle.tex

We will show that $G$ is not a polygon-circle graph.
Suppose $G$ has an intersection model in the set of circle filaments. 
Due to Cycle Lemma for Interval Filament Graphs, there is $i \in [2]$ such that
the path $v^i_0, \ldots, v^i_{8}$ is represented by a chain of overlapping
circle filaments -- see Figure~\ref{fig:interval_filament_not_polygon_circle} to the right.
Now, note that $\dom(b^i)$ is contained in the set $\dom(v^i_3) \cup \dom(v^i_4) \cup \dom(v^i_5)$.
However, the circle filament representing $a^i$ is disjoint with the circle filaments representing $v^i_3, v^i_4, v^i_5$.
It proves that the circle filaments representing $a^i$ and $b^i$ can not intersect, which is a contradiction.

Now we prove statement \eqref{enum:flat_grounded_L_graph_not_interval_filament}.
For an integer $n\geq 3$, consider the graph $G_n$ that consists of:
\begin{itemize}
 \item a cycle $C$ of size $2n$, whose consecutive vertices are denoted by $v_0$, $v_{\{0,1\}}$, $v_{1}$, $v_{\{1,2\}}$, $\ldots$, 
 $v_{n-1}$, $v_{\{n-1,0\}}$,
 \item a vertex $v_{\{i,i+1,i+2\}}$ defined for every $i \in [n-1]$, adjacent to the vertices $v_i,v_{i+1},v_{i+2}$ (we take indices modulo~$n$). 
 Note that $v_{\{0,1,2\}}$ is not defined.
\end{itemize} 
Note that $G_n$ is the intersection graph of a flat L-collection -- see Figure~\ref{fig:flat_grounded_L_graph_not_inteval_filament} for an illustration of $G_5$. 
\input ./figures/flat_grounded_L_graph_not_interval_filament.tex

For large enough $n$, the graph $G_n$ is not an interval filament graph.
Suppose for sake of contradiction that $G_n$ has an intersection model in the set of interval filaments for some $n\ge 9$.
First, consider the cycle $C$ in $G_n$. 
Due to Cycle Lemma for Interval Filament Graphs, in any representation of $C$ as an interval filament graph, 
there is a subpath $P$ of size $2n-4 > 13$ such that the vertices of $P$ are represented by
a chain of overlapping interval filaments.
Note that $P$ is long enough so that there exists $i \in [n-1]$ 
such that $v_i, v_{i+1}, v_{i+2}, v_{i+3}$ are the inner vertices of~$P$ and $v_{\{i, i+1, i+2\}}$ and $v_{\{i+1, i+2, i+3\}}$ are vertices of $G_n$.
Then by the the Cycle Lemma for Interval Filament Graphs, the left end point of $\dom(v_{\{i, i+1, i+2\}})$ would need to be contained in $\dom(v_i)$ and its right endpoint in $\dom(v_{i+2})$, 
and similarly the left end point of $\dom(v_{\{i+1, i+2, i+3\}})$ would need to be contained in $\dom(v_{i+1})$ and its right endpoint in $\dom(v_{i+3})$. 
However, then the interval filaments of $v_{\{i, i+1, i+2\}}$ and $v_{\{i+1, i+2, i+3\}}$ would intersect, a contradiction.

To show statement \eqref{enum:monotone_L_graph_not_interval_filament}, consider the graph $G$ that consists of:
\begin{itemize}
 \item a cycle $C$, whose consecutive vertices are denoted by $v^1_0,\ldots,v^1_{10}$, $v^2_0, \ldots, v^2_{10}$,
 \item four vertices $a^1,b^1,a^2,b^2$ such that for every $i \in [2]$ the vertex $a^i$ is adjacent to $b^i,v^i_1,v^i_2$ and 
 the vertex $b^i$ is adjacent to $a^i,v^i_4,v^i_5$.
\end{itemize}
See Figure~\ref{fig:monotone_L_graph_not_interval_filament} that shows $G$ as the intersection graph of monotone L-shapes.

\input ./figures/monotone_L_graph_not_interval_filament.tex

We will show that $G$ is not an interval filament graph.
Suppose otherwise that $G$ has an intersection model in the set of interval filaments.
Due to Cycle Lemma for Interval Filament Graphs, we deduce that there 
is $i \in [2]$ such that the path $v^i_0,\ldots, v^i_6$
is represented by a chain of overlapping interval filaments -- see Figure~\ref{fig:monotone_L_graph_not_interval_filament} to the right.
Again, by the same lemma we imply that $\dom(a^i)$ is contained in $\dom(v^i_1) \cup \dom(v^i_2)$ and $\dom(b^i)$ is contained in $\dom(v^i_4) \cup \dom(v^i_5)$.
However, since these sets are disjoint, the interval filaments representing $a^i$ and $b^i$ are disjoint as well, 
which is a contradiction.

To prove statement \eqref{enum:polygon_circle_not_outer_1_string}, consider the graph $G$ that consists of:
\begin{itemize}
 \item a cycle $C$, whose consecutive vertices are denoted by $v^0_0,\ldots,v^0_5$, $v^1_0,\ldots,v^1_5$, $\ldots$,
 $v^4_0,\ldots,v^4_5$,
 \item two adjacent vertices $a$ and $b$ such that $a$ is adjacent to $v^i_0,v^i_4$ for every $i \in [0,4]$
 and $b$ is adjacent to $v^i_2$ for every $i \in [0,4]$.
\end{itemize}
Note that $G$ is a polygon-circle graph -- see Figure~\ref{fig:polygon_circle_not_outer_1_string} for an illustration of $G$.
\input ./figures/polygon_circle_not_outer_1_string.tex

We will show that $G$ is not an outer-$1$-string graph.
Suppose $G$ is the intersection graph of a simple family of grounded curves $\mathcal{F}$. 
Let $I$ be the independent set in $G$, consisting of the vertices $\{v^i_0,v^i_2,v^i_4: i \in [4]\}$.
Cycle Lemma for Outer-1-string Graphs asserts that the order of the bases of the strings representing 
the vertices of $I$ coincides with its order in the cycle $C$.
It follows that there is $i \in [0,4]$ such that either
\begin{itemize}
\item $b(v^j_0)< b(v^j_2) < b(v^j_4)$ for $j \in [i,i+3]$ and $[b(v^j_0),b(v^j_4)] <  [b(v^{j+1}_0),b(v^{j+1}_4)]$ for $j \in [i,i+2]$, or
\item $b(v^j_4) < b(v^j_2) < b(v^j_0)$ for $j \in [i,i+3]$ and $[b(v^{j+1}_4),b(v^{j+1}_0)] <  [b(v^{j}_4),b(v^{j}_0)]$ for $j \in [i,i+2]$.
\end{itemize}
Again, the superscripts are taken modulo $4$.
Suppose the first case holds.
Clearly, there are two indices $j<k$ in $[i,i+3]$ such that both the bases of $s(a)$ and of $s(b)$ 
are outside of the intervals $[b(v^j_0),b(v^j_4)]$ and $[b(v^k_0),b(v^k_4)]$.
Consider a Jordan curve $J(j)$ that consists of:
\begin{itemize}
\item a substring of $s(v^j_0)$ between $b(v^j_0)$ and the 
intersection point of $s(v^j_0)$ and $s(a)$,
\item a substring of $s(a)$ between the intersection points of $s(a)$ with the strings
$s(v^j_0)$ and $s(v^j_4)$,
\item a substring of $s(v^j_4)$ between $b(v^j_4)$ and the 
intersection point of $s(v^j_4)$ and $s(a)$,
\item the segment of the grounding line between $b(v^j_0)$ and $b(v^j_4)$.
\end{itemize}
Note that the interior of $s(v^j_2)$ is contained in the interior of $J(j)$ -- see Figure~\ref{fig:polygon_circle_not_outer_1_string} for an illustration.
We define a Jordan curve $J(k)$ similarly and we note that the interior of $s(v^k_2)$ is contained in the interior of $J(k)$.
Since the strings $s(v^j_0),s(v^j_4),s(v^k_0),s(v^k_4)$ are pairwise disjoint and $[b(v^j_0), b(v^j_4)] < [b(v^k_0),b(v^k_4)]$, 
the interiors of $J(j)$ and $J(k)$ are also disjoint (despite the fact that $J(j)$ and $J(k)$ may touch each other).
Now, consider the string $s(b)$.
Note that the base of $s(b)$ is contained in the exteriors of $J(i)$ and $J(k)$.
Now, since $s(b)$ intersects $s(v^j_2)$ and $s(v^j_4)$, we deduce that $s(b)$ must enter the interiors of $J(j)$ and of $J(k)$.
Since $s(b)$ is disjoint with $s(v^j_0),s(v^j_4),s(v^k_0),s(v^k_4)$, we deduce that $s(b)$ intersects $s(a)$ at least two times, 
once when entering the interior of $J(j)$ and the second time when entering the interior of $J(k)$.
This is a contradiction with the fact that $\mathcal{F}$ is a simple family of grounded strings. 
\end{proof}

\bibliographystyle{plain}
\bibliography{L-shapes}

\end{document}

%% file: figures/circle_graphs.tex
\begin{figure}[h]
\centering
\begin{tikzpicture}[xscale=0.2,yscale=0.2,>=latex]

\begin{scope}[shift={(-38,0)}]
\draw[dashed,-] (-8,0) -- (8,0);
\coordinate (center) at (0,4) {};
\draw[thick, brown] ($(center)+(150:4cm)$) -- ($(center)+(20:4cm)$);
\draw[thick, blue] ($(center)+(210:4cm)$) -- ($(center)+(340:4cm)$);

\draw[thick, red] ($(center)+(120:4cm)$) -- ($(center)+(240:4cm)$);
\draw[thick, green] ($(center)+(70:4)$) -- ($(center)+(300:4cm)$);

\draw[thick,|->] ([shift=(180:4cm)]0,4) arc (180:540:4cm);
\end{scope}

\begin{scope}[shift={(-19,0)}]
\draw[thick,|->] (-8,0) -- (8,0);
\draw[thick,red] ([shift=(0:5cm)]0,0) arc (0:180:5cm);
\draw[thick,green] ([shift=(0:2.75cm)]0,0) arc (0:180:2.75cm);
\draw[thick,blue] ([shift=(0:2.75cm)]-4,0) arc (0:180:2.75cm);
\draw[thick,brown] ([shift=(0:2.75cm)]4,0) arc (0:180:2.75cm);
\end{scope}

\begin{scope}[shift={(0,0)}]
\draw[thick,|->] (-8,0) -- (8,0);
\draw[thick,red] (-5,0) -- (0,5) -- (5,0);
\draw[thick,green] (-3,0) -- (0,3) -- (3,0);
\draw[thick,blue] (-7,0) -- (-4,3) -- (-1,0);
\draw[thick,brown] (1,0) -- (4,3) -- (7,0);
\end{scope}

\begin{scope}[shift={(19,0)},yscale=0.5,xscale=0.85]
\draw[thick,dashed,->] (-8,0) -- (8.5,16.5);
\draw[thick,-] (-9,0) -- (9,0);
\draw[thick,red] (-5,0) -- (-5,13) -- (5,13);
\draw[thick,green] (-3,0) -- (-3,11) -- (3,11);
\draw[thick,blue] (-7,0) -- (-7,7) -- (-1,7);
\draw[thick,brown] (1,0) -- (1,15) -- (7,15);
\end{scope}

\end{tikzpicture}
\caption{Different representations of a circle graph. A circle graph as a grounded $L$-graph.}
\label{fig:circle_graphs}
\end{figure}

%% file: figures/permutation_interval_graphs.tex
\begin{figure}[h]
\centering
\begin{tikzpicture}[xscale=0.2,yscale=0.2,>=latex]

\begin{scope}[shift={(-29,0)}]
\draw[dashed,-] (-7,0) -- (7,0);
\draw[thick,-] (-5,6) -- (5,6);
\draw[thick,-] (-5,0) -- (5,0);

\draw[red,thick,-] (-3,0) -- (-1,6);

\draw[blue,thick,] (-1,0) -- (3,6);

\draw[green,thick,-] (1,0) -- (-3,6);

\draw[brown,thick,-] (3,0) -- (1,6);

\end{scope}

\begin{scope}[shift={(-11,0)}]
\draw[thick,-] (-8.5,0) -- (8.5,0);
\draw[dashed,-] (0,0) -- (0,8);

\draw[thick,red] ([shift=(0:6.5cm)]-1,0) arc (0:180:6.5cm);

\draw[thick,green] ([shift=(0:5.5cm)]2,0) arc (0:180:5.5cm);

\draw[thick,blue] ([shift=(0:3.5cm)]-2,0) arc (0:180:3.5cm);

\draw[thick,brown] ([shift=(0:2.5cm)]1,0) arc (0:180:2.5cm);

\end{scope}

\begin{scope}[shift={(11,0)}]
\draw[thick,-] (-8,0) -- (8,0);

\draw[red,thick,|-|] (-5,1.5) -- (2,1.5);
\draw[thick,red,dotted,-] (-5,1.5) -- (-5,0);
\draw[thick,red,dotted,-] (2,1.5) -- (2,0);

\draw[blue,thick,|-|] (-3,3) -- (0,3);
\draw[thick,blue,dotted,-] (-3,3) -- (-3,0);
\draw[thick,blue,dotted,-] (0,3) -- (0,0);

\draw[green,thick,|-|] (-1,4.5) -- (4,4.5);
\draw[thick,green,dotted,-] (-1,4.5) -- (-1,0);
\draw[thick,green,dotted,-] (4,4.5) -- (4,0);

\draw[brown,thick,|-|] (3,3) -- (5,3);
\draw[thick,brown,dotted,-] (3,3) -- (3,0);
\draw[thick,brown,dotted,-] (5,3) -- (5,0);

\end{scope}

\begin{scope}[shift={(29,4)},yscale=0.5]
\draw[thick,-] (-8,-8) -- (8,8);
\draw[thick,-] (-8,-8) -- (8,-8);

\draw[red,thick,-] (-4,-8) -- (-4,-4) -- (3,-4);

\draw[blue,thick,-] (-2,-8) -- (-2,-2) -- (1,-2);

\draw[green,thick,-] (0,-8) -- (0,0) -- (5,0);

\draw[brown,thick,-] (4,-8) -- (4,4) -- (6,4);

\end{scope}

\end{tikzpicture}
\caption{To the left: different representations of a permutation graph. To the right: an interval graph as a monotone $L$-graph and as a grounded $L$-graph.}
\label{fig:permutation_interval_graphs}
\end{figure}

%% file: figures/polygon_circle_graphs.tex
\begin{figure}[h]
\centering
\begin{tikzpicture}[xscale=0.2,yscale=0.2,>=latex]

\begin{scope}[shift={(-19,0)}]
\draw[dashed,-] (-8,0) -- (8,0);
\coordinate (center) at (0,4) {};

\coordinate (P1) at ($(center)+(210:4cm)$) {};
\coordinate (P2) at ($(center)+(240:4cm)$) {};
\coordinate (P3) at ($(center)+(270:3.96cm)$) {};
\coordinate (P4) at ($(center)+(300:4cm)$) {};
\coordinate (P5) at ($(center)+(330:4cm)$) {};
\coordinate (P6) at ($(center)+(0:4cm)$) {};
\coordinate (P7) at ($(center)+(30:4cm)$) {};
\coordinate (P8) at ($(center)+(60:4cm)$) {};
\coordinate (P9) at ($(center)+(90:3.92cm)$) {};
\coordinate (P10) at ($(center)+(120:4cm)$) {};
\coordinate (P11) at ($(center)+(150:3.85cm)$) {};

\begin{scope}[fill opacity=0.5]
\draw[fill=green!30] (P3)--(P5)--(P6)--(P9)--cycle;
\draw[fill=red!30] (P7)--(P8)--(P11)--cycle;
\end{scope}
\draw[brown, thick] (P1)--(P4);
\draw[blue, thick] (P2)--(P10);
\draw[green, thick] (P3)--(P5)--(P6)--(P9)--cycle;
\draw[red, thick] (P7)--(P8)--(P11)--cycle;

\draw[thick,|->] ([shift=(180:4cm)]0,4) arc (180:540:4cm);
\end{scope}

\begin{scope}[shift={(0,0)}, xscale=2,yscale=2]
\draw[thick,|->] (-4,0) -- (4.2,0);
\draw[very thick,white] (-4,3) -- (4,3);
\draw[thick, blue] ([shift=(0:3cm)]0,0) arc (0:180:3cm);
\draw[thick, brown] ([shift=(0:1.25cm)]-2.25,0) arc (0:180:1.25cm);
\draw[thick, green] ([shift=(0:0.75cm)]-0.75,0) arc (0:180:0.75cm);
\draw[thick, green] ([shift=(0:0.5cm)]0.5,0) arc (0:180:0.5cm);
\draw[thick, green] ([shift=(0:0.75cm)]1.75cm,0) arc (0:180:0.75cm);
\draw[thick, red] ([shift=(0:0.25cm)]1.75cm,0) arc (0:180:0.25cm);
\draw[thick, red] ([shift=(0:0.75cm)]2.75cm,0) arc (0:180:0.75cm);
\end{scope}

\begin{scope}[shift={(25,0)},xscale=1.5,yscale=1.5]
\draw[dashed,-] (-8,0) -- (7,0);
\draw[red, thick,|-|] (-6,-0.2)--(4,-0.2);
\draw[red,thick] ([shift=(0:1cm)]-5,0) arc (0:180:1cm);
\draw[red,thick] ([shift=(0:1cm)]-3,0) arc (0:180:1cm);
\draw[red,thick] ([shift=(0:1cm)]-1,0) arc (0:180:1cm);
\draw[red,thick] ([shift=(0:1cm)]1,0) arc (0:180:1cm);
\draw[red,thick] ([shift=(0:1cm)]3,0) arc (0:180:1cm);

\draw[green, thick,|-|] (-5,-0.4)--(1,-0.4);
\draw[green,thick] ([shift=(0:1cm)]-4,0) arc (0:180:1cm);
\draw[green,thick] ([shift=(0:1cm)]-2,0) arc (0:180:1cm);
\draw[green,thick] ([shift=(0:1cm)]0,0) arc (0:180:1cm);

\draw[blue, thick,|-|] (3,-0.4)--(5,-0.4);
\draw[blue,thick] ([shift=(0:1cm)]4,0) arc (0:180:1cm);

\end{scope}
\end{tikzpicture}
\caption{To the left: a polygon-circle graph and its representation as an intersection graph of circle filaments. To the right: an interval graph as a polygon-circle graph.}
\label{fig:polygon_circle_graphs}
\end{figure}

%% file: figures/graph_classes_ext.tex
\begin{figure}[h]
\centering
\begin{tikzpicture}[xscale=0.45,yscale=0.55,>=latex]

\draw[thick,->] (-9.5,1)--(-9.5,2);
\draw[thick,->] (-9.5,4)--(-10.5,6);

\draw[thick,->] (-10.5,8)--(-10.5,9);
\draw[thick,->] (-10.5,11)--(-10.5,12);
\draw[thick,->] (-10.5,14)--(-10.5,15);
\draw[thick,->] (-8,17.03)--(-2,18.97);

\draw[thick,->] (-5.5,4)--(0,7.5);

\draw[thick,->] (6,2.5)--(4,7.5);

\draw[thick,->] (9,2.5)--(13.75,9);

\draw[thick,->] (2.5,9.5)--(2.5,12.5);

\draw[thick,->] (2.5,14.5)--(2.5,19);

\draw[thick,->] (13,13)--(6.5,19);

\draw[thick,->] (2.5,2.5)--(-5.5,6);

\begin{scope}[shift={(-15,-1)}]
\draw (0,0) rectangle (11,2);
\begin{scriptsize}
\node[above] at (4,1.0) {permutation graphs};
\node[above] at (4,0.2) {$\chi = \omega$};
\end{scriptsize}

\begin{scope}[shift={(8,0)}]
\draw (0,0.3) -- (2.5,0.3);
\draw[dashed] (0,1.7) -- (2.5,1.7);

\draw[red,thick] (0.5,0.3) -- (1,1.7);
\draw[green,thick] (1,0.3) -- (2,1.7);
\draw[blue,thick] (1.5,0.3) -- (0.5,1.7);
\draw[brown,thick] (2,0.3) -- (1.5,1.7);

\end{scope}
 \end{scope}

\begin{scope}[shift={(2,0.5)}]
\draw (0,0) rectangle (11,2);
\begin{tiny}
\node[above] at (4,1) {interval graphs};
\node[above] at (4,0.2) {$\chi = \omega$};
\end{tiny}

\begin{scope}[shift={(8,0)}]
\draw (0,0.3) -- (2.5,0.3);

\draw[red,thick,-] (0.3,0.7) -- (1,0.7);
\draw[red,dotted] (0.3,0.7) -- (0.3,0.3);
\draw[red,dotted] (1,0.7) -- (1,0.3);
\draw[thick,red] (0.3,0.6) -- (0.3,0.8);
\draw[thick,red] (1,0.6) -- (1,0.8);

\draw[green,thick,-] (1.3,0.7) -- (2.0,0.7);
\draw[green,dotted] (1.3,0.7) -- (1.3,0.3);
\draw[green,dotted] (2,0.7) -- (2.0,0.3);
\draw[thick,green] (1.3,0.6) -- (1.3,0.8);
\draw[thick,green] (2,0.6) -- (2.0,0.8);

\draw[blue,thick,-] (0.5,1.1) -- (1.5,1.1);
\draw[blue,dotted] (0.5,1.1) -- (0.5,0.3);
\draw[blue,dotted] (1.5,1.1) -- (1.5,0.3);
\draw[thick,blue] (0.5,1.0) -- (0.5,1.2);
\draw[thick,blue] (1.5,1.0) -- (1.5,1.2);

\draw[brown,thick,-] (1.7,1.1) -- (2.2,1.1);
\draw[brown,dotted] (1.7,1.1) -- (1.7,0.3);
\draw[brown,dotted] (2.2,1.1) -- (2.2,0.3);
\draw[thick,brown] (1.7,1) -- (1.7,1.2);
\draw[thick,brown] (2.2,1) -- (2.2,1.2);

\end{scope}
 \end{scope}

\begin{scope}[shift={(-15,2)}]
\draw (0,0) rectangle (11,2);
\begin{tiny}
\node[above] at (4,1) {circle graphs};
\node[above] at (4,0.1) {$\chi = \Theta(\omega \log{ \omega})$};
\end{tiny}
\begin{scope}[shift={(8,0)},xscale=1.2]

\coordinate (center) at (1.15,1) {};
\draw (center) circle (0.8cm);

\draw[thick, brown] ($(center)+(150:0.8cm)$) -- ($(center)+(20:0.8cm)$);
\draw[thick, blue] ($(center)+(210:0.8cm)$) -- ($(center)+(340:0.8cm)$);
\draw[thick, red] ($(center)+(120:0.8cm)$) -- ($(center)+(300:0.8cm)$);
\draw[thick, green] ($(center)+(60:0.8)$) -- ($(center)+(240:0.8cm)$);

\end{scope}
\end{scope}

\begin{scope}[shift={(-16,6)}]
\draw (0,0) rectangle (11,2);
\begin{tiny}
\node[above] at (4,1) {grounded $L$-graphs};
\node[above] at (4,0.1) {$\Omega(\omega \log{ \omega}) \leq \chi \leq O(\omega^4)$};
\end{tiny}
\begin{scope}[shift={(8,0)}]
\draw (0,0.3) -- (2.5,0.3);

\draw[thick,red] (0.5,0.3) -- (0.5,1)--(1.5,1);
\draw[thick,green] (1,0.3) -- (1,1.4)--(2,1.4);
\draw[thick,blue] (1.5,0.3) -- (1.5,0.7)--(1.9,0.7);
\draw[thick,brown] (1.8,0.3) -- (1.8,1.7)--(2.2,1.7);

\end{scope}
\end{scope}

\begin{scope}[shift={(-3,7.5)}]
\draw (0,0) rectangle (11,2);
\begin{tiny}
\node[above] at (4,1) {polygon-circle graphs};
\node[above] at (4,0.2) {$\Omega(\omega \log{ \omega}) \leq \chi \leq O(\omega^2)$};
\end{tiny}
\begin{scope}[shift={(8,0)},xscale=1.2]

\coordinate (center) at (1.15,1) {};

\coordinate (P1) at ($(center)+(210:0.8cm)$) {};
\coordinate (P2) at ($(center)+(240:0.8cm)$) {};
\coordinate (P3) at ($(center)+(245:0.8cm)$) {};
\coordinate (P4) at ($(center)+(300:0.8cm)$) {};
\coordinate (P5) at ($(center)+(330:0.8cm)$) {};
\coordinate (P6) at ($(center)+(0:0.8cm)$) {};
\coordinate (P7) at ($(center)+(30:0.8cm)$) {};
\coordinate (P8) at ($(center)+(60:0.8cm)$) {};
\coordinate (P9) at ($(center)+(90:0.8cm)$) {};
\coordinate (P10) at ($(center)+(120:0.8cm)$) {};
\coordinate (P11) at ($(center)+(170:0.75cm)$) {};

\begin{scope}[fill opacity=0.5]
\draw[fill=green!30] (P3)--(P5)--(P6)--(P9)--cycle;
\draw[fill=red!30] (P7)--(P8)--(P11)--cycle;
\draw[fill=blue!30] (P1)--(P4)--(P10);
\end{scope}
\draw[blue] (P1)--(P4)--(P10)--cycle;
\draw[green] (P3)--(P5)--(P6)--(P9)--cycle;
\draw[red] (P7)--(P8)--(P11)--cycle;

\draw (center) circle (0.8cm);
\end{scope}
\end{scope}

\begin{scope}[shift={(-3,12.5)}]
\draw (0,0) rectangle (11,2);
\begin{tiny}
\node[above] at (4,1) {interval filament graphs};
\node[above] at (4,0.1) {${\omega +1 \choose 2} \leq \chi \leq O(\omega^3\log{\omega})$};
\end{tiny}
\begin{scope}[shift={(8,0)}]
\draw (0,0.3) -- (2.5,0.3);

\draw[thick, red, rounded corners=4] (0.25,0.3) -- (0.27,0.75) -- (1.25,0.75)  -- (1.27,1.7) -- (1.9,1.7) -- (2,0.3);
\draw[thick, blue, rounded corners=4] (0.5,0.3) -- (0.6,1.2) -- (1.4,1) -- (1.5,0.3);
\draw[thick, brown, rounded corners=2] (0.75,0.3) -- (0.8,0.6) -- (1.2,0.6) -- (1.25,0.3);
\draw[thick, green, rounded corners=2] (1.75,0.3) -- (1.85,1) -- (2.15,1) -- (2.25,0.3);
\end{scope}
\end{scope}

\begin{scope}[shift={(-16,9)}]
\draw (0,0) rectangle (11,2);
\begin{tiny}
\node[above] at (4,1) {grounded $\{\mathrm{L},\text{\reflectbox{$\mathrm{L}$}}\}$-graphs};
\node[above] at (4,0.1) {$\Omega(\omega \log{ \omega}) \leq \chi \leq O(\omega^4)$};
\end{tiny}
\begin{scope}[shift={(8,0)}]
\draw (0,0.3) -- (2.5,0.3);
\draw[thick,red] (0.5,0.3) -- (0.5,1)--(1.5,1);
\draw[thick,green] (1,0.3) -- (1,1.4)--(2.2,1.4);
\draw[thick,blue] (1.5,0.3) -- (1.5,0.7)--(0.3,0.7);
\draw[thick,brown] (2,0.3) -- (2,1.7)--(1.5,1.7);

\end{scope}
\end{scope}

\begin{scope}[shift={(-16,12)}]
\draw (0,0) rectangle (11,2);
\begin{tiny}
\node[above] at (4,1) {grounded segment graphs};
\node[above] at (4,0.1) {$\Omega(\omega \log{ \omega}) \leq \chi \leq 2^{O(2^\omega)}$};
\end{tiny}
\begin{scope}[shift={(8,0)}]
\draw (0,0.3) -- (2.5,0.3);
\draw[thick,red] (0.5,0.3) --(1.5,1.5);
\draw[thick,green] (1,0.3) --(2.2,1.4);
\draw[thick,blue] (1.5,0.3)--(0.3,0.7);
\draw[thick,brown] (2,0.3)--(1,1.7);

\end{scope}
\end{scope}

\begin{scope}[shift={(-16,15)}]
\draw (0,0) rectangle (11,2);
\begin{tiny}
\node[above] at (4,1) {outer-$1$-string graphs};
\node[above] at (4,0.1) {$\Omega(\omega \log{ \omega}) \leq \chi \leq 2^{O(2^\omega)}$};
\end{tiny}
\begin{scope}[shift={(8,0)}]
\draw (0,0.3) -- (2.5,0.3);
\draw[thick,red,use Hobby shortcut] (0.5,0.3) .. (0.7,1.2) .. (1.5,1.1) .. (1.75,1.2);
\draw[thick,green,use Hobby shortcut] (1,0.3) .. (1.2,0.7) .. (1.5,0.7) .. (2.3,1.7);
\draw[thick,blue,use Hobby shortcut] (1.5,0.3).. (1.6,0.7) .. (1,1)..(0.7,1.6)..(0.25,1.6);
\draw[thick,brown,use Hobby shortcut] (2,0.3) .. (2.1,1).. (1.8,1.7).. (1.5,1);
\end{scope}
\end{scope}

\begin{scope}[shift={(-4,19)}]
\draw (0,0) rectangle (13,2);
\begin{tiny}
\node[above] at (5,1) {outerstring graphs};
\node[above] at (5,0.1) {${\omega +1 \choose 2} \leq \chi \leq 2^{O(2^{\omega(\omega-1)/ 2})}$};
\end{tiny}
\begin{scope}[shift={(10,0)}]
\draw (0,0.3) -- (2.5,0.3);
\draw[thick,red,use Hobby shortcut] (0.5,0.3) .. (0.5,0.7) .. (1.5,0.8) .. (2,1.7);
\draw[thick,green,use Hobby shortcut] (1,0.3) .. (1,1) .. (1,1.3) .. (1.5,1.3).. (2.25,1.3);
\draw[thick,blue,use Hobby shortcut] (1.5,0.3).. (1.5,0.8) .. (1.5,1.5) .. (1.3,1.7) .. (0.9,1.7)..(0.7,1.3)..(0.7,0.6);
\draw[thick,brown,use Hobby shortcut] (2,0.3) .. (2,0.4).. (1.8,0.5).. (0.8,0.5);
\end{scope}
\end{scope}

\begin{scope}[shift={(10,9)}]
\draw (0,0) rectangle (7.5,4);
\begin{tiny}
\node[above] at (2.25,2.5) {monotone};
\node[above] at (2.25,1.5) {$L$-graphs};
\node[above] at (2.4,0.5) {$\chi \leq O(\omega \log{\omega})$};
\end{tiny}
\begin{scope}[shift={(4.5,1.7)},yscale=0.68]
\draw (0,0) -- (2.5,2.5);
\draw[red,thick] (0.5,0) -- (0.5,0.5) --(2.3,0.5);
\draw[green,thick] (1,0.65) -- (1,1) --(2.2,1);
\draw[blue,thick] (1.5,0.2) -- (1.5,1.5) --(1.8,1.5);
\draw[brown,thick] (2,0.75) -- (2,2) --(2.4,2);

\end{scope}
\end{scope}

\end{tikzpicture}
\caption{Hierarchy of inclusions between graph classes considered in the introduction. An arrow indicates inclusion.}
\label{fig:graph_classes_hierarchy}
\end{figure}

%% file: figures/pillar.tex
\begin{figure}
    \centering
    \begin{tikzpicture}[xscale=0.85,>=latex,shorten >=-0.2pt]
        \def \h {.75}
        \drawLUltra{-3}{\h}{5}
        \drawL{-1}{\h*2}{5}
        \drawL{-7}{\h*2.5}{5}
        \drawL{1}{\h*3}{4}
        \drawL{-8}{\h*3.5}{2}
        \drawLUltra{-5}{\h*4}{7.5}
        \node[label=below:$b$] (b0) at (0,0) {};
        \draw[ultra thick, red, dashed, ->] (b0.center) --++ (0,\h) --++ (-3,0) --++ (0,3*\h) --++ (-2, 0) --++ (0,\h);
        \draw (-8.5,0)--(5.5,0);
    \end{tikzpicture}
    \caption{A pillar $P_b$ with base $b$, where $P_b$ is in red and the supports of $P_b$ is bolded.}
    \label{fig:pillar}
\end{figure}
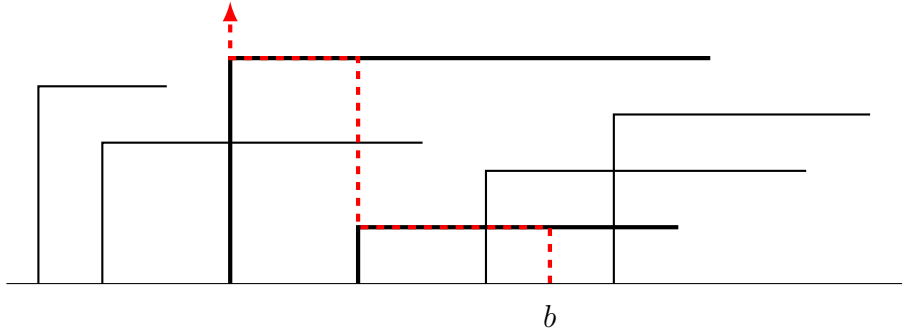

%% file: figures/pillar_assignment_and_pillar_coloring.tex
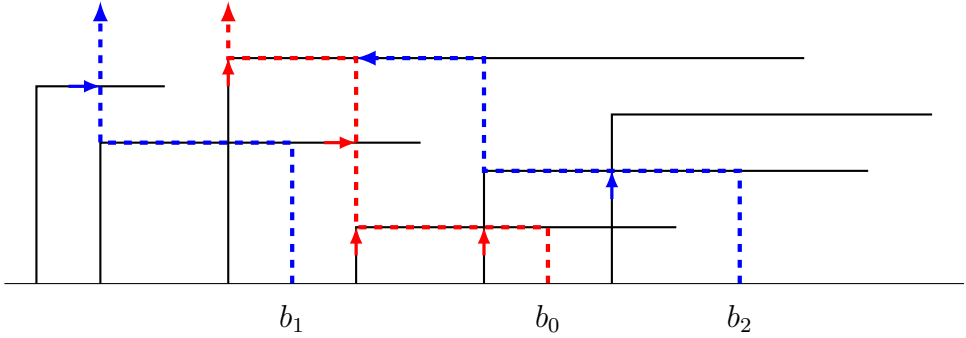
\begin{figure}
    \centering
    \begin{tikzpicture}[xscale=0.85,>=latex,shorten >=-0.2pt]
        \def \h {.75}
        \drawLColoured{-3}{\h}{5}{black}
        \drawLColoured{-1}{\h*2}{6}{black}
        \drawLColoured{-7}{\h*2.5}{5}{black}
        \drawLColoured{1}{\h*3}{5}{black}
        \drawLColoured{-8}{\h*3.5}{2}{black}
        \drawLColoured{-5}{\h*4}{9}{black}
        
        \draw[red, very thick, ->] (-5, \h*3.5) -- (-5, \h*4);
        \draw[blue, very thick, ->] (-7.5, \h*3.5) -- (-7, \h*3.5);
        \draw[red, very thick, ->] (-3.5, \h*2.5) -- (-3, \h*2.5);
        \draw[red, very thick, ->] (-3, \h*0.5) -- (-3, \h*1);
        \draw[red, very thick, ->] (-1, \h*0.5) -- (-1, \h*1);
        \draw[blue, very thick, ->] (1, \h*1.5) -- (1, \h*2);
        \node[label=below:$b_0$] (b0) at (0,0) {};
        \node[label=below:$b_1$] (b1) at (-4,0) {};
        \node[label=below:$b_2$] (b2) at (3,0) {};
        \draw[ultra thick, red, dashed, ->] (b0.center) --++ (0,\h) --++ (-3,0) --++ (0,3*\h) --++ (-2, 0) --++ (0,\h);
        \draw[ultra thick, blue, dashed, ->] (b1.center) --++ (0,\h*2.5) --++ (-3,0) --++ (0,\h*2.5);
        \draw[ultra thick, blue, dashed,->] (b2.center) --++ (0,\h*2) --++ (-4,0) --++ (0,\h*2) --++ (-2,0);
        \draw (-8.5,0)--(6.5,0);
    \end{tikzpicture}
    \caption{A pillar assignment and corresponding coloring $b_0, b_1, b_2$ with $b_0 \prec b_1 \prec b_2$.
    An arrow on every $L$-shape indicates its intersection with the smallest pillar.}
    \label{fig:pillarPhi}
\end{figure}

%% file: figures/degrees.tex
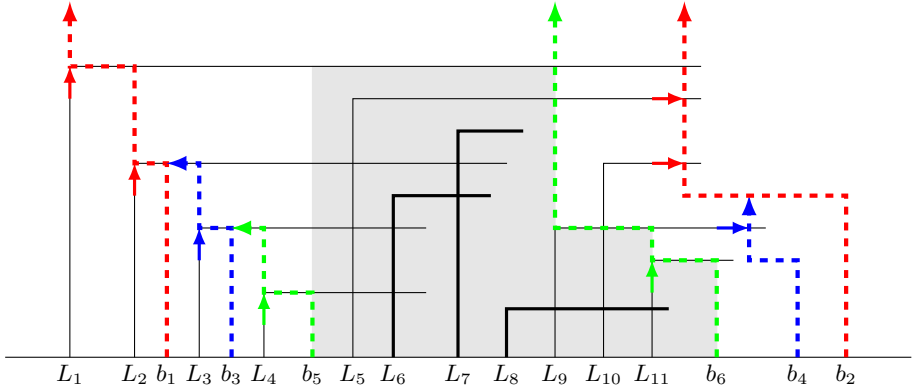
\begin{figure}
    \centering
    \begin{tikzpicture}[scale=0.43,>=latex,shorten >=-0.2pt]
        \draw[color=gray!20,fill=gray!20] (-4.5,0)--(-4.5,9)--(3,9)--(3,4)--(6,4)--(6,3)--(8,3)--(8,0)--cycle;
        \draw (-14,0)--(14,0);
        
        \draw (-12,0)--(-12,9)--(7.5,9);
        \draw[red, very thick, ->] (-12,8)--(-12,9);
        
        \draw (-10,0)--(-10,6)--(1.5,6);
        \draw[red, very thick, ->] (-10,5)--(-10,6);
        
        \draw (-8,0)--(-8,4)--(-1,4);
        \draw[blue, very thick, ->] (-8,3)--(-8,4);

        \draw (-6,0)--(-6,2)--(-1,2);
        \draw[green, very thick, ->] (-6,1)--(-6,2);
        
        \draw (-3.25,0)--(-3.25,8)--(7.5,8);
        \draw[red, very thick, ->] (6,8) -- (7,8);

        \draw[very thick] (-2,0)--(-2,5)--(1,5);
        \draw[very thick] (0,0)--(0,7)--(2,7);
        \draw[very thick] (1.5,0)--(1.5,1.5)--(6.5,1.5);

        \draw (3,0)--(3,4)--(9.5,4);
        \draw[blue, very thick, ->] (8,4) -- (9,4);
        
        \draw (4.5,0)--(4.5,6)--(7.5,6);
        \draw[red, very thick, ->] (6,6) -- (7,6);

        \draw (6,0)--(6,3)--(8.5,3);
        \draw[green, very thick, ->] (6,2) -- (6,3);

        \draw[ultra thick, red, dashed, ->] (-9,0) -- (-9,6) -- (-10,6) -- (-10,9)--(-12,9)--(-12,11);
        \draw[ultra thick, blue, dashed, ->] (-7,0) -- (-7,4) -- (-8,4) -- (-8,6)--(-9,6);
        \draw[ultra thick, green, dashed, ->] (-4.5,0) -- (-4.5,2) -- (-6,2) -- (-6,4)--(-7,4);

        \draw[ultra thick, green, dashed, ->] (8,0) -- (8,3) --(6,3) -- (6,4) -- (3,4)--(3,11);
        \draw[ultra thick, blue, dashed, ->] (10.5,0) -- (10.5,3) -- (9,3) -- (9,5);
        \draw[ultra thick, red, dashed, ->] (12,0) -- (12,5) -- (7,5) -- (7,11);

\begin{scriptsize}        
        \node[label=below:$b_1$] (b1) at (-9,0.3) {};
        \node[label=below:$b_3$] (b3) at (-7,0.3) {};
        \node[label=below:$b_5$] (b5) at (-4.5,0.3) {};
        \node[label=below:$b_2$] (b2) at (12,0.3) {};
        \node[label=below:$b_4$] (b4) at (10.5,0.3) {};
        \node[label=below:$b_6$] (b6) at (8,0.3) {};

        \node[label=below:$L_1$] (L1) at (-12,0.3) {};
        \node[label=below:$L_2$] (L2) at (-10,0.3) {};
        \node[label=below:$L_3$] (L3) at (-8,0.3) {};
        \node[label=below:$L_4$] (L4) at (-6,0.3) {};
        \node[label=below:$L_5$] (L5) at (-3.25,0.3) {};
        \node[label=below:$L_6$] (L6) at (-2,0.3) {};
        \node[label=below:$L_7$] (L7) at (0,0.3) {};
        \node[label=below:$L_8$] (L8) at (1.5,0.3) {};
        \node[label=below:$L_9$] (L9) at (3,0.3) {};
        \node[label=below:$L_{10}$] (L10) at (4.5,0.3) {};
        \node[label=below:$L_{11}$] (L11) at (6,0.3) {};
\end{scriptsize}        
        
        \end{tikzpicture}
    \caption{We assume $b_1 \prec b_2 \prec \ldots \prec b_6$.
    The pillar $P_{b_3}$ belongs to the left neighbourhood of $(b_5,b_6)$, which is witnessed by $L_3$ and $L_6$.
    The pillar $P_{b_4}$ belongs to the right neighbourhood of $(b_5,b_6)$, which is witnessed by $L_9$.
    The left neighbourhood of $(b_5,b_6)$ consists of $P_{b_1},P_{b_3},P_{b_5}$, the right neighbourhood of $(b_5,b_6)$ 
    consists of $P_{b_2},P_{b_4},P_{b_6}$.
    A~gray area shows a region where unassigned $L$-shapes from $(b_5,b_6)$ must be contained.
    } 
    \label{fig:pillardeg}
\end{figure}

%% file: figures/cascading.tex
\begin{figure}
    \centering
    \begin{tikzpicture}[scale=0.43,>=latex,shorten >=-0.2pt]
        \draw (-14,0)--(14,0);
        
        \draw[ultra thick] (-12,0)--(-12,9)--(-8,9);
        \draw[orange, very thick, ->] (-11.5,9)--(-10.5,9);
        
        \draw[ultra thick] (-9,0)--(-9,8)--(-6,8);
        \draw[green, very thick, ->] (-8,8)--(-7,8);
        
        \draw[ultra thick] (-6,0)--(-6,7)--(-4,7);
        \draw[blue, very thick, ->] (-5.5,7)--(-4.5,7);

        \draw[ultra thick] (-3,0)--(-3,6)--(12,6);
        \draw[red, very thick, ->] (-3,5)--(-3,6);
        
        \draw[ultra thick] (-4.5,0)--(-4.5,4.5)--(9,4.5);
        \draw[blue, very thick, ->] (-4.5,3.5)--(-4.5,4.5);
        
        \draw[ultra thick] (-7,0)--(-7,3)--(6,3);
        \draw[green, very thick, ->] (-7,2)--(-7,3);
        
        \draw[ultra thick] (-10.5,0)--(-10.5,1.5)--(3,1.5);
        \draw[orange, very thick, ->] (-10.5,0.5)--(-10.5,1.5);

        \draw[ultra thick, red, dashed, ->] (10.5,0) -- (10.5,6) -- (-3,6) -- (-3,11);
        \draw[ultra thick, blue, dashed, ->] (7.5,0) -- (7.5,4.5) -- (-4.5,4.5) -- (-4.5,11);
        \draw[ultra thick, green, dashed, ->] (4.5,0) -- (4.5,3) -- (-7,3) -- (-7,11);
        \draw[ultra thick, orange, dashed, ->] (1.5,0) -- (1.5,1.5) -- (-10.5,1.5) -- (-10.5,11);

\begin{scriptsize}        
        \node[label=below:$b_1$] (b1) at (1.5,0.3) {};
        \node[label=below:$b_3$] (b3) at (7.5,0.3) {};
        \node[label=below:$b_2$] (b2) at (4.5,0.3) {};
        \node[label=below:$b_4$] (b4) at (10.5,0.3) {};

        \node[label=below:$L_1$] (L1) at (-12,0.3) {};
        \node[label=below:$L_2$] (L2) at (-9,0.3) {};
        \node[label=below:$L_3$] (L3) at (-6,0.3) {};
        \node[label=below:$L_4$] (L4) at (-3,0.3) {};
        \node[label=below:$L_1^*$] (L5) at (-10.5,0.3) {};
        \node[label=below:$L_2^*$] (L6) at (-7,0.3) {};
        \node[label=below:$L_3^*$] (L7) at (-4.5,0.3) {};
\end{scriptsize}        
        
        \end{tikzpicture}
    \caption{Cascading $L$-shapes $L_1,L_2,L_3,L_4$, and one possible configuration of pairwise intersecting $L$-shapes $L_1^*, L_2^*, L_3^*, L_4^*=L_4$ as in the conclusion of Lemma~\ref{clique}.
    } 
    \label{fig:cascad}
\end{figure}
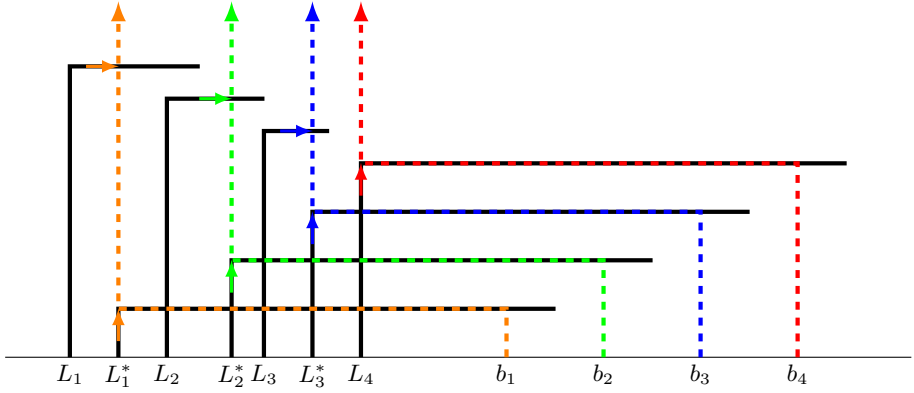

%% file: figures/cycle_lemma_for_interval_filaments_1.tex
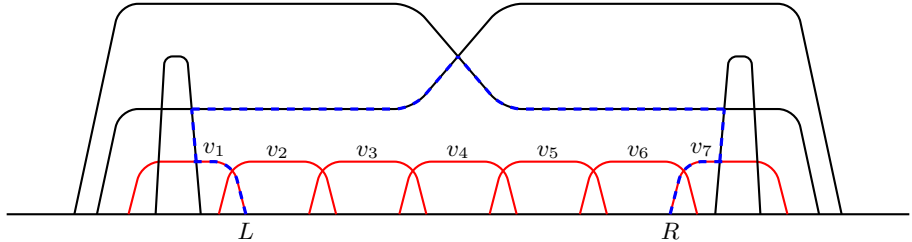
\begin{figure}[h]
\centering
\begin{tikzpicture}[xscale=0.6,yscale=0.7,>=latex]
\draw[thick] (-10,0)--(10,0);

\draw[red,thick,rounded corners=7] (-7.3,0)--(-7,1)--(-5,1)--(-4.7,0);
\draw[red,thick,rounded corners=7] (-5.3,0)--(-5,1)--(-3,1)--(-2.7,0);
\draw[red,thick,rounded corners=7] (-3.3,0)--(-3,1)--(-1,1)--(-0.7,0);
\draw[red,thick,rounded corners=7] (-1.3,0)--(-1,1)--(1,1)--(1.3,0);
\draw[red,thick,rounded corners=7] (0.7,0)--(1,1)--(3,1)--(3.3,0);
\draw[red,thick,rounded corners=7] (2.7,0)--(3,1)--(5,1)--(5.3,0);
\draw[red,thick,rounded corners=7] (4.7,0)--(5,1)--(7,1)--(7.3,0);

\draw[thick,rounded corners=3] (-6.7,0)--(-6.5,3)--(-6,3)--(-5.7,0);
\draw[thick,rounded corners=3] (5.7,0)--(6.0,3)--(6.5,3)--(6.7,0);

\draw[thick,rounded corners=7] (-8,0)--(-7.5,2)--(-1,2)--(1,4)--(7.5,4)--(8.5,0);

\draw[thick,rounded corners=7] (8,0)--(7.5,2)--(1,2)--(-1,4)--(-7.5,4)--(-8.5,0);


\draw[very thick, blue, dashed, rounded corners=7] (-4.7,0) -- (-5,1) --(-6,1);
\draw[very thick, blue, dashed] (-5.8,1) -- (-5.9,2);
\draw[very thick, blue, dashed, rounded corners=7] (-5.9,2)--(-1,2)--(0,3);
\draw[very thick, blue, dashed, rounded corners=7] (5.9,2)--(1,2)--(0,3);
\draw[very thick, blue, dashed] (5.8,1) -- (5.9,2);
\draw[very thick, blue, dashed, rounded corners=7] (4.7,0) -- (5,1) --(6,1);


\begin{scriptsize}
\node[above] at (-5.4,0.95) {$v_1$};
\node[above] at (-4,0.9) {$v_2$};
\node[above] at (-2,0.9) {$v_3$};
\node[above] at (0,0.9) {$v_4$};
\node[above] at (2,0.9) {$v_5$};
\node[above] at (4,0.9) {$v_6$};
\node[above] at (5.4,0.95) {$v_7$};

\node[below] at (-4.7,0) {$L$};
\node[below] at (4.7,0) {$R$};
\end{scriptsize}
\end{tikzpicture}
\caption{A path $P=v_1\ldots v_7$ in $C$ (depicted in red) is represented by a chain of overlapping interval filaments.
A curve $J$ is drawn with a dashed blue line.}
\label{fig:cycle_lemma_for_interval_filaments_1}
\end{figure}

%% file: figures/cycle_lemma_for_interval_filaments_2.tex
\begin{figure}[h]
\centering
\begin{tikzpicture}[xscale=0.6,yscale=0.7,>=latex]
\draw[thick] (-10,0)--(10,0);

\draw[thick,rounded corners=7] (-7.3,0)--(-7,1)--(-5,1)--(-4.7,0);
\draw[thick,rounded corners=7] (-5.3,0)--(-5,1)--(-3,1)--(-2.7,0);
\draw[thick,rounded corners=7,dashed] (-3.3,0)--(-3,1)--(-1,1)--(-0.7,0);
\draw[thick,rounded corners=7,dashed] (-1.3,0)--(-1,1)--(1,1)--(1.3,0);
\draw[thick,rounded corners=7,dashed] (0.7,0)--(1,1)--(3,1)--(3.3,0);
\draw[thick,rounded corners=7] (2.7,0)--(3,1)--(5,1)--(5.3,0);
\draw[thick,rounded corners=7] (4.7,0)--(5,1)--(7,1)--(7.3,0);

\draw[thick,rounded corners=3] (-6.7,0)--(-6.5,3)--(-6,3)--(-5.7,0);
\draw[thick,rounded corners=3] (5.7,0)--(6.0,3)--(6.5,3)--(6.7,0);

\draw[thick,rounded corners=7] (-8.5,0)--(-7.5,2.5)--(-1,2.5)--(1,4)--(7.5,4)--(8.5,0);
\draw[thick,rounded corners=7, dashed] (-7.9,0)--(-7,1.75)--(-1,1.75)--(1,2.5)--(7.2,2.5)--(7.9,0);



\begin{scriptsize}
\node[above] at (-5.4,0.95) {$w_2$};
\node[above] at (-4,0.9) {$w_3$};
\node[above] at (4,0.9) {$w_{n-5}$};
\node[below] at (7.8,0) {$w_{n-4}$};

\node[above] at (-6.25,2.9) {$w_1$};
\node[above] at (6.25,2.9) {$w_{n-3}$};

\node[above] at (3.5,3.95) {$w'$};
\node[above] at (3.5,2.45) {$w''$};

\end{scriptsize}
\end{tikzpicture}
\caption{}
\label{fig:cycle_lemma_for_interval_filaments_2}
\end{figure}

%% file: figures/interval_filament_not_polygon_circle.tex
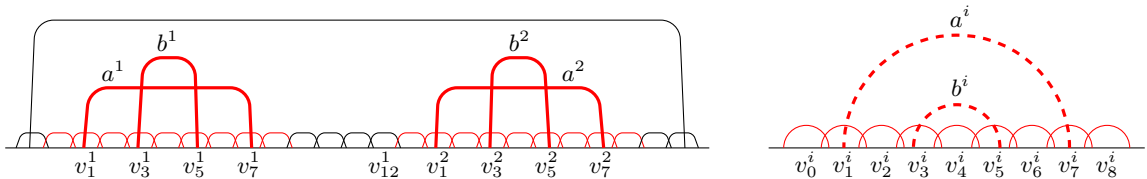
\begin{figure}[h]
\centering
\begin{tikzpicture}[xscale=0.18,yscale=0.1,>=latex]
\begin{scope}
\draw[-] (-26,0) -- (26,0);
\draw[-,white] (-24,-5) -- (24,-5);
\draw[rounded corners=2] (-25.2,0) -- (-24.8,2) --(-23.2,2)--(-22.8,0);
\draw[red,rounded corners=2] (-23.2,0) -- (-22.8,2) --(-21.2,2)--(-20.8,0);
\draw[red,rounded corners=2] (-21.2,0) -- (-20.8,2) --(-19.2,2)--(-18.8,0);
\draw[red,rounded corners=2] (-19.2,0) -- (-18.8,2) --(-17.2,2)--(-16.8,0);
\draw[red,rounded corners=2] (-17.2,0) -- (-16.8,2) --(-15.2,2)--(-14.8,0);
\draw[red,rounded corners=2] (-15.2,0) -- (-14.8,2) --(-13.2,2)--(-12.8,0);
\draw[red,rounded corners=2] (-13.2,0) -- (-12.8,2) --(-11.2,2)--(-10.8,0);
\draw[red,rounded corners=2] (-11.2,0) -- (-10.8,2) --(-9.2,2)--(-8.8,0);
\draw[red,rounded corners=2] (-9.2,0) -- (-8.8,2) --(-7.2,2)--(-6.8,0);
\draw[red,rounded corners=2] (-7.2,0) -- (-6.8,2) --(-5.2,2)--(-4.8,0);
\draw[rounded corners=2] (-5.2,0) -- (-4.8,2) --(-3.2,2)--(-2.8,0);
\draw[rounded corners=2] (-3.2,0) -- (-2.8,2) --(-1.2,2)--(-0.8,0);
\draw[rounded corners=2] (-1.2,0) -- (-0.8,2) --(0.8,2)--(1.2,0);
\draw[rounded corners=2] (0.8,0) -- (1.2,2) --(2.8,2)--(3.2,0);
\draw[red,rounded corners=2] (2.8,0) -- (3.2,2) --(4.8,2)--(5.2,0);
\draw[red,rounded corners=2] (4.8,0) -- (5.2,2) --(6.8,2)--(7.2,0);
\draw[red,rounded corners=2] (6.8,0) -- (7.2,2) --(8.8,2)--(9.2,0);
\draw[red,rounded corners=2] (8.8,0) -- (9.2,2) --(10.8,2)--(11.2,0);
\draw[red,rounded corners=2] (10.8,0) -- (11.2,2) --(12.8,2)--(13.2,0);
\draw[red,rounded corners=2] (12.8,0) -- (13.2,2) --(14.8,2)--(15.2,0);
\draw[red,rounded corners=2] (14.8,0) -- (15.2,2) --(16.8,2)--(17.2,0);
\draw[red,rounded corners=2] (16.8,0) -- (17.2,2) --(18.8,2)--(19.2,0);
\draw[red,rounded corners=2] (18.8,0) -- (19.2,2) --(20.8,2)--(21.2,0);
\draw[rounded corners=2] (20.8,0) -- (21.2,2) --(22.8,2)--(23.2,0);
\draw[rounded corners=2] (22.8,0) -- (23.2,2) --(24.8,2)--(25.2,0);

\draw[rounded corners=7] (-24.2,0) -- (-23.8,17) --(23.8,17)--(24.2,0);

\draw[red,very thick,rounded corners=7] (-20.2,0) -- (-19.8,8) --(-8,8)--(-7.8,0);
\draw[red,very thick,rounded corners=7] (-16.2,0) -- (-15.8,12) --(-12,12)--(-11.8,0);

\draw[red,very thick,rounded corners=7] (18.2,0) -- (17.8,8) --(6,8)--(5.8,0);
\draw[red,very thick,rounded corners=7] (14.2,0) -- (13.8,12) --(10,12)--(9.8,0);

\begin{scriptsize}

\node[above] at (-18,7.7) {$a^1$};
\node[above] at (-14,11.7) {$b^1$};

\node[above] at (16,7.7) {$a^2$};
\node[above] at (12,11.7) {$b^2$};

\node at (-20,-2.2) {$v^1_1$};
\node at (-16.0,-2.2) {$v^1_3$};
\node at (-12,-2.2) {$v^1_5$};
\node at (-8,-2.2) {$v^1_7$};
\node at (2,-2.2) {$v^1_{12}$};

\node at (6,-2.2) {$v^2_1$};
\node at (10.0,-2.2) {$v^2_3$};
\node at (14,-2.2) {$v^2_5$};
\node at (18,-2.2) {$v^2_7$};

\end{scriptsize}
\end{scope}
\end{tikzpicture}
\hspace{0.5cm}
\begin{tikzpicture}[xscale=0.25,yscale=0.25,>=latex]
\draw[-] (-10,0) -- (10,0);
\draw[-,white] (-10,-2) -- (10,-2);
\draw[red] ([shift=(0:1.2cm)]-8,0) arc (0:180:1.2cm);
\draw[red] ([shift=(0:1.2cm)]-6,0) arc (0:180:1.2cm);
\draw[red] ([shift=(0:1.2cm)]-4,0) arc (0:180:1.2cm);
\draw[red] ([shift=(0:1.2cm)]-2,0) arc (0:180:1.2cm);
\draw[red] ([shift=(0:1.2cm)]-0,0) arc (0:180:1.2cm);
\draw[red] ([shift=(0:1.2cm)]2,0) arc (0:180:1.2cm);
\draw[red] ([shift=(0:1.2cm)]4,0) arc (0:180:1.2cm);
\draw[red] ([shift=(0:1.2cm)]6,0) arc (0:180:1.2cm);
\draw[red] ([shift=(0:1.2cm)]8,0) arc (0:180:1.2cm);

\draw[red,very thick, dashed] ([shift=(0:6cm)]0,0) arc (0:180:6cm);
\draw[red,very thick, dashed] ([shift=(0:2.3cm)]0,0) arc (0:180:2.3cm);

\begin{scriptsize}
\node[below] at (-8,0.25) {$v^i_0$};
\node[below] at (-6.0,0.25) {$v^i_1$};
\node[below] at (-4.0,0.25) {$v^i_2$};
\node[below] at (-2.0,0.25) {$v^i_3$};
\node[below] at (0,0.25) {$v^i_4$};
\node[below] at (2,0.25) {$v^i_5$};
\node[below] at (4,0.25) {$v^i_6$};
\node[below] at (6,0.25) {$v^i_7$};
\node[below] at (8,0.25) {$v^i_8$};

\node[above] at (0.2,6.0) {$a^i$};
\node[above] at (0.2,2.3) {$b^i$};

\end{scriptsize}
\end{tikzpicture}

\caption{An interval filament graph which is not a polygon-circle graph.}
\label{fig:interval_filament_not_polygon_circle}
\end{figure}

%% file: figures/flat_grounded_L_graph_not_interval_filament.tex
\begin{figure}[h]
\centering
\begin{tikzpicture}[xscale=0.31,yscale=0.15,>=latex]
\draw[thick] (-16,0)--(10,0);
\draw[thick,white] (-16,-3)--(10,-3);

\draw[red] (0,0)--(0,6)--(7,6);
\draw[red] (-3,0)--(-3,9)--(1,9);
\draw[red] (-6,0)--(-6,12)--(3,12);
\draw[red] (-9,0)--(-9,15)--(5,15);
\draw[red] (-12,0)--(-12,18)--(7,18);

\draw[red] (-4,0)--(-4,1)--(0.5,1);
\draw[red] (-7,0)--(-7,4)--(-2.5,4);
\draw[red] (-10,0)--(-10,7)--(-5.5,7);
\draw[red] (-13,0)--(-13,10)--(-8.5,10);

\draw[very thick] (-8,0)--(-8,5)--(0.5,5);
\draw[very thick] (-11,0)--(-11,8)--(-2.5,8);
\draw[very thick] (-14,0)--(-14,11)--(-5.5,11);

\draw[very thick] (4,0)--(4,20)--(8,20);
\draw[red] (6,0)--(6,19)--(9,19);

\begin{scriptsize}
\node[below] at (0,0) {$v_1$};
\node[below] at (-3,0) {$v_0$};
\node[below] at (-6,0) {$v_4$};
\node[below] at (-9,0) {$v_3$};
\node[below] at (-12,0) {$v_2$};

\node[below] at (3.75,0) {$v_{\{1,2,3\}}$};
\node[below] at (6.75,0) {$v_{\{1,2\}}$};

\end{scriptsize}
\end{tikzpicture}
\hspace{0.5cm}
\begin{tikzpicture}[xscale=0.16,yscale=0.15,>=latex]
\draw[thick] (-19.5,0)--(19.5,0);
\draw[white] (-19,-3)--(19,-3);

\draw[red,rounded corners=6] (-18.5,0) -- (-18,6) --(-14,6)--(-13.5,0);
\draw[red,rounded corners=6] (-14.5,0) -- (-14,6) --(-10,6)--(-9.5,0);
\draw[red,rounded corners=6] (-10.5,0) -- (-10,6) --(-6,6)--(-5.5,0);
\draw[red,rounded corners=3] (-6.5,0) -- (-4,12) --(-2,12)--(-1.5,0);
\draw[red,rounded corners=6] (-2.5,0) -- (-2,6) --(2,6)--(2.5,0);
\draw[red,rounded corners=3] (1.5,0) -- (2,12) --(4,12)--(6.5,0);
\draw[red,rounded corners=6] (5.5,0) -- (6,6) --(10,6)--(10.5,0);
\draw[red,rounded corners=6] (9.5,0) -- (10,6) --(14,6)--(14.5,0);
\draw[red,rounded corners=6] (13.5,0) -- (14,6) --(18,6)--(18.5,0);

\draw[very thick, dashed, rounded corners=6] (-12,0) -- (-11.5,10) --(3.5,10)--(4,0);
\draw[very thick, dashed, rounded corners=6] (-4,0) -- (-3.5,9) --(11.5,9)--(12,0);

\begin{scriptsize}
\node[above] at (-13,5.6) {$v_{i}$};
\node[above] at (-3,11.6) {$v_{i+1}$};
\node[above] at (3,11.6) {$v_{i+2}$};
\node[above] at (14,5.4) {$v_{i+3}$};

\node[above] at (-10,9.7) {$v_{\{i,i+1,i+2\}}$};
\node[above] at (11,8.7) {$v_{\{i+1,i+2,i+3\}}$};

\end{scriptsize}
\end{tikzpicture}

\caption{To the left: a flat grounded L-graph $G_5$. To the right: a~part of $G_{10}$ as an interval filament graph.}
\label{fig:flat_grounded_L_graph_not_inteval_filament}
\end{figure}
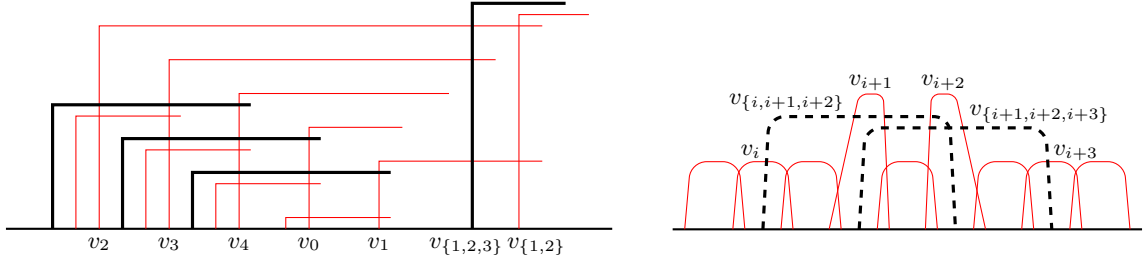

%% file: figures/monotone_L_graph_not_interval_filament.tex
\begin{figure}[h]
\centering
\begin{tikzpicture}[xscale=0.2,yscale=0.1,>=latex]
\draw[thick] (-0,-0)--(-43.5,-43.5);

\draw (-3,-5.5)--(-3,-3)--(-0.5,-3);
\draw[red] (-5,-7.5)--(-5,-5)--(-2.5,-5);

\draw[red] (-7,-9.5)--(-7,-7)--(-4.5,-7);
\draw[very thick,red] (-8,-15)--(-8,-8)--(-6,-8);
\draw[red] (-9,-11.5)--(-9,-9)--(-6.5,-9);

\draw[red] (-11,-13.5)--(-11,-11)--(-8.5,-11);

\draw[red] (-13,-15.5)--(-13,-13)--(-10.5,-13);
\draw[very thick,red] (-14,-16)--(-14,-14)--(-7.5,-14);
\draw[red] (-15,-17.5)--(-15,-15)--(-12.5,-15);

\draw[red] (-17,-19.5)--(-17,-17)--(-14.5,-17);
\draw (-19,-21.5)--(-19,-19)--(-16.5,-19);
\draw (-21,-23.5)--(-21,-21)--(-18.5,-21);
\draw (-23,-25.5)--(-23,-23)--(-20.5,-23);
\draw (-25,-27.5)--(-25,-25)--(-22.5,-25);
\draw[red] (-27,-29.5)--(-27,-27)--(-24.5,-27);

\draw[red] (-29,-31.5)--(-29,-29)--(-26.5,-29);
\draw[very thick,red] (-30,-37)--(-30,-30)--(-28,-30);
\draw[red] (-31,-33.5)--(-31,-31)--(-28.5,-31);

\draw[red] (-33,-35.5)--(-33,-33)--(-30.5,-33);
\draw[red] (-35,-37.5)--(-35,-35)--(-32.5,-35);
\draw[very thick,red] (-36,-38)--(-36,-36)--(-29.5,-36);
\draw[red] (-37,-39.5)--(-37,-37)--(-34.5,-37);
\draw[red] (-39,-41.5)--(-39,-39)--(-36.5,-39);
\draw (-41,-43.5)--(-41,-41)--(-38.5,-41);

\draw (-1,-43)--(-1,-1)--(0,-1);
\draw (-42.5,-43.5)--(-42.5,-42.5)--(-0.5,-42.5);


\begin{scriptsize}
\node[above] at (-37.3,-37.3) {$v^1_1$};
\node[above] at (-35.3,-35.3) {$v^1_2$};
\node[above] at (-31.3,-31.3) {$v^1_4$};
\node[above] at (-29.3,-29.3) {$v^1_5$};

\node[right] at (-30.3,-34) {$b^1$};
\node[below] at (-32.5,-35.5) {$a^1$};

\node[above] at (-15.3,-15.3) {$v^2_1$};
\node[above] at (-13.3,-13.3) {$v^2_2$};
\node[above] at (-9.3,-9.3) {$v^2_4$};
\node[above] at (-7.3,-7.3) {$v^2_5$};

\node[right] at (-8.3,-12) {$b^2$};
\node[below] at (-10.5,-13.5) {$a^2$};

\end{scriptsize}
\end{tikzpicture}
\hspace{0.5cm}
\begin{tikzpicture}[xscale=0.3,yscale=0.6,>=latex]
\draw[thick,white] (-8,-0.2)--(8,-0.2);
\draw[thick] (-8,0)--(8,0);
\draw[red,rounded corners=4] (4.8,0) -- (5,1) --(7,1)--(7.2,0);
\draw[red,rounded corners=4] (2.8,0) -- (3,1) --(5,1)--(5.2,0);
\draw[red,rounded corners=4] (0.8,0) -- (1,1) --(3,1)--(3.2,0);
\draw[red,rounded corners=4] (-1.2,0) -- (-1,1) --(1,1)--(1.2,0);
\draw[red,rounded corners=4] (-3.2,0) -- (-3,1) --(-1,1)--(-0.8,0);
\draw[red,rounded corners=4] (-5.2,0) -- (-5,1) --(-3,1)--(-2.8,0);
\draw[red,rounded corners=4] (-7.2,0) -- (-7,1) --(-5,1)--(-4.8,0);

\draw[red,very thick,rounded corners=4, dashed] (2.2,0) -- (2.5,3) --(3.5,3)--(3.8,0);
\draw[red,very thick,rounded corners=4, dashed] (-3.8,0) -- (-3.5,3) --(-2.5,3)--(-2.2,0);

\begin{scriptsize}
\node[above] at (-2.8,2.9) {$a^i$};
\node[above] at (3.2,2.9) {$b^i$};
\node[above] at (-6,0.9) {$v^i_0$};
\node[above] at (-4.4,0.9) {$v^i_1$};
\node[above] at (-1.55,0.9) {$v^i_2$};
\node[above] at (0,0.9) {$v^i_3$};
\node[above] at (1.6,0.9) {$v^i_4$};
\node[above] at (4.45,0.9) {$v^i_5$};
\node[above] at (6,0.9) {$v^i_6$};
\end{scriptsize}
\end{tikzpicture}

\caption{A monotone $L$-graph which is not an interval filament graph.}
\label{fig:monotone_L_graph_not_interval_filament}
\end{figure}
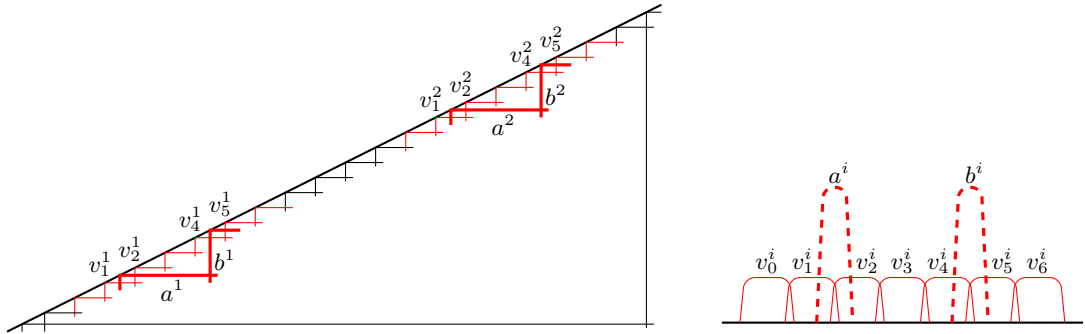

%% file: figures/polygon_circle_not_outer_1_string.tex
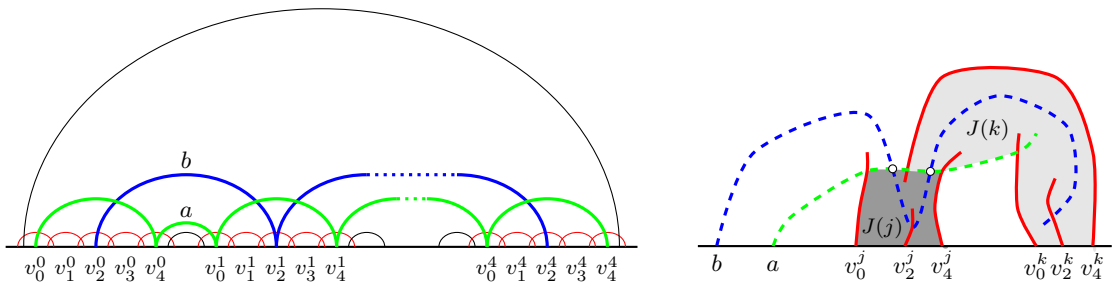
\begin{figure}[h]
\centering
\begin{tikzpicture}[xscale=0.2,yscale=0.16,>=latex]
\begin{scope}
\draw[-,thick] (-24,0) -- (18,0);
\draw[-,white] (-24,-4) -- (18,-4);
\draw[red] ([shift=(0:1.2cm)]-22,0) arc (0:180:1.2cm);
\draw[red] ([shift=(0:1.2cm)]-20,0) arc (0:180:1.2cm);
\draw[red]([shift=(0:1.2cm)]-18,0) arc (0:180:1.2cm);
\draw[red] ([shift=(0:1.2cm)]-16,0) arc (0:180:1.2cm);
\draw[red] ([shift=(0:1.2cm)]-14,0) arc (0:180:1.2cm);
\draw ([shift=(0:1.2cm)]-12,0) arc (0:180:1.2cm);

\draw[red] ([shift=(0:1.2cm)]-10,0) arc (0:180:1.2cm);
\draw[red] ([shift=(0:1.2cm)]-8,0) arc (0:180:1.2cm);
\draw[red] ([shift=(0:1.2cm)]-6,0) arc (0:180:1.2cm);
\draw[red] ([shift=(0:1.2cm)]-4,0) arc (0:180:1.2cm);
\draw[red] ([shift=(0:1.2cm)]-2,0) arc (0:180:1.2cm);
\draw ([shift=(0:1.2cm)]0,0) arc (0:180:1.2cm);

\draw ([shift=(0:1.2cm)]6,0) arc (0:180:1.2cm);
\draw[red] ([shift=(0:1.2cm)]8,0) arc (0:180:1.2cm);
\draw[red] ([shift=(0:1.2cm)]10,0) arc (0:180:1.2cm);
\draw[red] ([shift=(0:1.2cm)]12,0) arc (0:180:1.2cm);
\draw[red] ([shift=(0:1.2cm)]14,0) arc (0:180:1.2cm);
\draw[red] ([shift=(0:1.2cm)]16,0) arc (0:180:1.2cm);

\draw ([shift=(0:19.8cm)]-3,0) arc (0:180:19.8cm);

\draw[very thick, blue] ([shift=(0:6cm)]-12,0) arc (0:180:6cm);
\draw[very thick, blue] ([shift=(90:6cm)]0,0) arc (90:180:6cm);
\draw[very thick, dotted, blue] (0,6)--(6,6);
\draw[very thick, blue] ([shift=(0:6cm)]6,0) arc (0:90:6cm);

\draw[very thick, green] ([shift=(0:4cm)]-18,0) arc (0:180:4cm);
\draw[very thick, green] ([shift=(0:2cm)]-12,0) arc (0:180:2cm);
\draw[very thick, green] ([shift=(0:4cm)]-6,0) arc (0:180:4cm);
\draw[very thick, green] ([shift=(90:4cm)]2,0) arc (90:180:4cm);
\draw[very thick, dotted, green] (2,4)--(4,4);
\draw[very thick, green] ([shift=(0:4cm)]4,0) arc (0:90:4cm);
\draw[very thick, green] ([shift=(0:4cm)]12,0) arc (0:180:4cm);

\begin{scriptsize}
\node[above] at (-12,5.8) {$b$};
\node[above] at (-12,1.8) {$a$};
\end{scriptsize}
\begin{tiny}

\node at (-22,-1.8) {$v^0_0$};
\node at (-20.0,-1.8) {$v^0_1$};
\node at (-18.0,-1.8) {$v^0_2$};
\node at (-16.0,-1.8) {$v^0_3$};
\node at (-14.0,-1.8) {$v^0_4$};

\node at (-10,-1.8) {$v^1_0$};
\node at (-8.0,-1.8) {$v^1_1$};
\node at (-6.0,-1.8) {$v^1_2$};
\node at (-4,-1.8) {$v^1_3$};
\node at (-2,-1.8) {$v^1_4$};

\node at (16,-1.8) {$v^4_4$};
\node at (14.0,-1.8) {$v^4_3$};
\node at (12.0,-1.8) {$v^4_2$};
\node at (10.0,-1.8) {$v^4_1$};
\node at (8.0,-1.8) {$v^4_0$};

\end{tiny}
\end{scope}
\end{tikzpicture}
\hspace{0.5cm}
\begin{tikzpicture}[xscale=1,yscale=1,>=latex]
\draw[-,white] (-2.75,-0.65) -- (2.75,-0.65);

\path[name path=g] (-2.75,0)--(2.75,0);

\fill[black!40] plot[smooth,tension=.7] coordinates {(-0.65,0) (0.5,0) (0.5,0) (0.4,0.55) (0.5,1) (0.5,1) (-0.3,1.02) (-0.5,1) (-0.5,1) (-0.6,0.5) (-0.65,0)}--cycle;
\draw[-] (-2.75,0) -- (2.75,0);

\fill[black!10] plot[smooth,tension=.7] coordinates {(1.75,0) (2.5,0) (2.5,0) (2.5,1.25) (2,2.25) (0.5,2.25) (0.05,1.0) (0.05,1.0) (0.5,1) (1.5,1.25) (1.5,1.25) (1.5,0.5) (1.75,0)}--cycle;
\draw[-,thick] (-2.75,0) -- (2.75,0);

\path[name path=vj0] plot[smooth,tension=.7] coordinates {(-0.65,0) (-0.6,0.5) (-0.5,1) (-0.5,1.25)};
\draw[very thick, red] plot[smooth,tension=.7] coordinates {(-0.65,0) (-0.6,0.5) (-0.5,1) (-0.5,1.25)};

\path[name path=vj2, thick, red] plot[smooth,tension=.7] coordinates {(0,0) (0.1,0.25)  (0.1,0.5)};
\draw[very thick, red] plot[smooth,tension=.7] coordinates {(0,0) (0.1,0.25)  (0.1,0.5)};

\path[name path=vj4] plot[smooth,tension=.7] coordinates {(0.5,0) (0.4,0.5)  (0.5,1)  (0.75,1.25)};
\draw[very thick, red] plot[smooth,tension=.7] coordinates {(0.5,0) (0.4,0.5)  (0.5,1)  (0.75,1.25)};

\path[name path=vk0] plot[smooth,tension=.7] coordinates {(1.75,0) (1.5,0.5)  (1.5,1.25)  (1.5,1.27)};
\draw[very thick, red] plot[smooth,tension=.7] coordinates {(1.75,0) (1.5,0.5)  (1.5,1.25)  (1.52,1.5)};

\path[name path=vk2] plot[smooth,tension=.7] coordinates {(2.1,0) (2,0.3)  (1.9,0.6)  (2,0.9)};
\draw[very thick, red] plot[smooth,tension=.7] coordinates {(2.1,0) (2,0.3)  (1.9,0.6)  (2,0.9)};

\path[name path=vk4] plot[smooth,tension=.7] coordinates {(2.5,0) (2.5,1.25)  (2,2.25)  (0.5,2.25) (0,0.85)};
\draw[red,very thick] plot[smooth,tension=.7] coordinates {(2.5,0) (2.5,1.25)  (2,2.25)  (0.5,2.25) (0,0.85)};

\path[name path=a, thick,] plot[smooth,tension=.7] coordinates {(-1.75,0) (-1.5,0.4) (-0.5,1) (0.5,1) (1.5,1.25) (1.75,1.5)};
\draw[green, very thick, dashed] plot[smooth,tension=.7] coordinates {(-1.75,0) (-1.5,0.4) (-0.5,1) (0.5,1) (1.5,1.25) (1.75,1.5)};

\path[name path=b] plot[smooth,tension=.7] coordinates              {(-2.5,0) (-2,1.25) (-0.5,1.75) (0.1,0.25) (0.5,1.5) (1.25,2) (2.15,1.6) (2.25,0.75) (1.8,0.25)};
\draw[blue, very thick, dashed] plot[smooth,tension=.7] coordinates {(-2.5,0) (-2,1.25) (-0.5,1.75) (0.1,0.25) (0.5,1.5) (1.25,2) (2.15,1.6) (2.25,0.75) (1.8,0.25)};

\draw[fill=white,name intersections={of=a and b}] (intersection-1) circle (1.5pt);
\draw[fill=white,name intersections={of=a and b}] (intersection-2) circle (1.5pt);



\begin{scriptsize}
\node[above] at (-1.75,-0.45) {$a$};
\node[above] at (-2.5,-0.45) {$b$};
\node[above] at (-0.65,-0.55) {$v^{j}_0$};
\node[above] at (0,-0.55) {$v^{j}_2$};
\node[above] at (0.5,-0.55) {$v^{j}_4$};
\node[above] at (1.75,-0.55) {$v^{k}_0$};
\node[above] at (2.1,-0.55) {$v^{k}_2$};
\node[above] at (2.5,-0.55) {$v^{k}_4$};
\end{scriptsize}
\begin{tiny}
\node[above] at (-0.3,0) {$J(j)$};
\node[above] at (1.1,1.3) {$J(k)$};
 
\end{tiny}

\end{tikzpicture}

\caption{A polygon-circle graph which is not an outer-$1$-string graph.}
\label{fig:polygon_circle_not_outer_1_string}
\end{figure}